\documentclass[12pt]{article}
\usepackage{amssymb}
\usepackage{amsmath}
\usepackage{amsthm}

\numberwithin{equation}{section}
\usepackage{graphicx} 
\newtheorem{theorem}{Theorem}[section]
\newtheorem{lemma}{Lemma}[section]
\newtheorem{proposition}{Proposition}[section]
\newtheorem{definition}{Definition}[section]

\newtheorem{remark}{Remark}[section]
\usepackage{epstopdf}
\usepackage{float}
\usepackage{epsf,epsfig,subfigure}
\date{}
\begin{document}
	\title{\bf Propagation Phenomena for A Reaction-Advection-Diffusion Competition Model in A Periodic Habitat}
	
	\author{Xiao Yu and Xiao-Qiang Zhao \thanks{Research supported in part by the NSERC of
				Canada.\,  Corresponding author.}\\
		Department of Mathematics and Statistics\\
		Memorial University of Newfoundland\\
		St. John's, NL A1C  5S7, Canada\\
		E-mail:\, xy3267@mun.ca \, \,  zhao@mun.ca}
	\maketitle
	\baselineskip 0.24in
	\noindent {\bf Abstract.}  This paper is devoted to the study of propagation phenomena for a Lotka-Volterra reaction-advection-diffusion competition model in a periodic habitat. We first investigate the global attractivity of a semi-trival steady state for the periodic initial value problem. Then we establish the existence of the rightward spreading speed and its coincidence with the minimal wave speed for spatially periodic rightward traveling waves. We also obtain a set of sufficient conditions for the rightward spreading speed to be linearly determinate. Finally, we apply the obtained results to a prototypical reaction-diffusion model.
	
	\smallskip
	

	\smallskip
	
\section{Introduction}
Over the past decade, there have been a number of works concerning about traveling waves and spreading speeds in  heterogeneous media, see, e.g., \cite{Xin} and references therein. More specifically,  G\"{a}tner and Friedlin \cite{Fri,GF} studied the spreading speed for an equation of Fisher type in which the mobility and the growth function vary periodically in space via probabilistic methods. Shigesada et al. \cite{SKT} first discussed the spread of a single species for a reaction-diffusion model in a patchy habitat with the periodic mobility and growth rate (see also \cite{SK}). Later, Berestycki, Hamel and Roques \cite{Hamel2, Hamel3} analyzed the following reaction-diffusion model in the periodically fragmented environment: 
\begin{equation}\label{BHR}
u_t-\nabla\cdot(A(x)\nabla u)=f(x,u), \quad x\in\mathbb{R}^N,
\end{equation}
where $A(x)$ and $f(x,u)$ depend on $x=(x_1,..,x_N)$ in a periodic fashion, and obtained the existence of pulsating waves and a variational formula for the minimal wave speed. A general theory of spreading speeds and traveling waves 
in a periodic habitat was developed by Weinberger \cite{Wein02} for a recursion with a periodic order-preserving compact operator, and by Liang and Zhao \cite{Liang2} for monotone semiflows with $\alpha$-contraction compactness. Weng and Zhao \cite{Weng} proposed a nonlocal and time-delayed reaction-diffusion model in a periodic habitat and studied 
its propagation phenomena by appealing to the abstract results in \cite{Liang2}. 
It is worthy to point out that the theory in \cite{Wein02,Liang2} may not apply to scalar evolution equations with nonlocal dispersal in a periodic habitat since the associated solution maps are not compact. Recently, 
 Shen and Zhang \cite{Shen1, Shen2} and Coville, D\'{a}vila and  Mart\'{i}nez \cite{CDM} investigated spreading speeds and periodic traveling waves for a large class of such equations via quite different approaches.

For two species reaction-diffusion competition models in a spatially homogeneous environment, there have been quite a few works on persistence, biological invasions of species, traveling wave solutions and the minimal wave speeds, see, e.g., \cite{Kan2,Hos,Lewis,Huang,Huang2,GL} and  references therein. In particular, Lewis, Li and Weinberger \cite{Lewis} studied the spreading speed of the two-species Lotka-Volterra competition model and gave a set of sufficient conditions for its linear determinacy. Huang \cite{Huang} and Guo and Liang \cite{GL} concerned about the minimal speed and the linear determinacy for more general cases. Huang and Han \cite{Huang2} further showed that the conjecture of linear determinacy is not true in general. Meanwhile, for a spatially heterogeneous environment, Dockery et al. \cite{DHMP} investigated the effect of different diffusion rates on the survival of two phenotypes of a species, and showed that the phenotype with the slower diffusion rate wins the competition. Recently, Lam and Ni \cite{LN} studied the global dynamics of two speices Lotka-Volterra competition-diffusion model with spatial heterogeneous growth rates in a bounded domain. Moreover,  Lutscher, McCauley and Lewis \cite{LML} added the advection term into such a competition model to discuss spatial patterns and coexistence mechanisms for stream populations.  However, it seems that there is no research on the propagation phenomena for two species reaction-diffusion competition model in a periodic habitat, which is the simplest form of the heterogeneous environment. 
   
 The purpose of this paper is to study the spatial dynamics of a more general two species competition  reaction-advection-diffusion model in a periodic habitat:
 \begin{eqnarray}\label{VL}
    & &\frac{\partial u_1}{\partial t}=L_1u_1+u_1(b_1(x)-a_{11}(x)u_1-a_{12}(x)u_2), \\ 
    & &\frac{\partial u_2}{\partial t}=L_2u_2+u_2(b_2(x)-a_{21}(x)u_1-a_{22}(x)u_2),\quad t>0, \ x\in\mathbb{R}. \nonumber
    \end{eqnarray}
 Here $L_iu=d_i(x)\frac{\partial^2 u}{\partial x^2}-g_i(x)\frac{\partial u}{\partial x}, i=1,2$, $u_1$ and $u_2$ 
 denote the population densities of two competing species in an $L$-periodic habitat for some positive 
 number $L$, $d_i(x)$, $g_i(x)$ and $b_i(x)$ are  diffusion, advection and growth rates of 
 the $i$-th species ($i=1,2$), respectively, 
 and $a_{ij}(x)(1\le i,j\le2)$ are inter- and intra-specific competition coefficients. We first establish the existence of two semi-trivial periodic steady states $(u^*_1(x),0)$ and $(0,u^*_2(x))$, and the global stability of  $(u^*_1(x),0)$ for system \eqref{VL} with periodic initial data.  Since the steady state $(0,0)$ is between $(u^*_1(x),0)$ and $(0,u^*_2(x))$ with respect to the competitive ordering, we cannot directly use the theory developed in \cite{Liang2} for monotone semiflows to study spreading speeds and spatially periodic traveling waves. Recently, Fang and Zhao \cite{FZ} investigated traveling waves for monotone semiflows with weak compactness in the case where there may be boundary fixed points between two ordered unstable and stable fixed points. Accordingly, in 
 the application of this theory one needs to determine whether the given system admits a single spreading speed and to identify the fixed points connected by traveling waves. Further, the abstract results in \cite{FZ} may not directly apply to the case of a periodic habitat. In Appendix, we adapt this theory for such a case  by combining the abstract results in \cite{Liang2} and \cite{FZ}. We then prove the existence of the rightward spatially periodic traveling waves of system \eqref{VL} connecting $(u^*_1(x),0)$ to $(0,u^*_2(x))$, and show that system \eqref{VL} admits a single rightward spreading speed via the method of upper solutions under appropriate assumptions. We also obtain a set of sufficient conditions for the rightward spreading speed to be linearly determinate. Since one more spreading speed is defined differently from the classical one, it is highly nontrivial to prove that those two speeds are identical. 
 
 The rest of this paper is organized as follows. In section 2, we first obtain the existence of two semi-trivial periodic steady states and the global stability of one semi-trivial periodic steady state for system \eqref{VL} with periodic initial data. In section 3, we establish the existence of the minimal wave speed of the rightward spatially periodic traveling waves and its coincidence with the minimal rightward spreading speed. In section 4, we show that the rightward spreading speed is linearly determinate under additional conditions. In section 5,  we apply the obtained results to a prototypical class of reaction-diffusion systems, which were studied in \cite{DHMP,LN} in the case of a bounded domain.  In the Appendix, we present the abstract results on traveling waves and spreading speeds for 
 monotone semiflows in a periodic habitat.

 After having submitted this paper for publication, we got knowledge of
 Kong, Rawal and Shen's recent paper \cite{KRS}, where they
 studied spreading speeds and linear determinacy for two species 
 competition systems with nonlocal dispersal in time and space periodic 
 habitats by using different methods from ours.

\section{The periodic initial value problem}
In this section, we investigate the global dynamics of the spatially periodic Lotka-Volterra competition 
system with the periodic initial values. 
  
 Throughout this paper, we assume that
 $d_i(x)$, $g_i(x)$, $a_{ij}(x)$ and $b_i(x)$ are $L$-periodic functions, $d_i, g_i,  a_{ij}, b_i\in C^\nu(\mathbb{R})$, and $a_{ij}(\cdot)>0$, $1\le i, j\le2,$ where $C^\nu(\mathbb{R})$ is a H\"{o}lder continuous space with the H\"{o}lder exponent $\nu\in(0,1)$; there exists a positive number $\alpha_0$ such that $d_i(x)\ge\alpha_0, \forall x\in\mathbb{R},i=1,2$, i.e., the operator $L_iu=d_i(x)\frac{\partial^2 u}{\partial x^2}-g_i(x)\frac{\partial u}{\partial x}$ is uniformly elliptic.   

Let $Y$ be the set of all continuous and $L$-periodic functions from $\mathbb{R}$ to $\mathbb{R}$, and $Y_+=\{\psi\in Y: \ \psi(x)\ge0,\forall x\in\mathbb{R}\}$ be a positive cone of $Y$. Equip $Y$ with the maximum norm $\|\phi\|_Y$, that is,
$\|\phi\|_Y=\max_{x\in \mathbb{R}}|\phi(x)|.$ Then $(Y,Y_+)$ is a strongly ordered Banach lattice.
Assume that $L$-periodic functions $d, g, h\in C^\nu(\mathbb{R})$ and $d(\cdot)>0$. 
It then follows that the scalar periodic eigenvalue problem 
\begin{eqnarray}\label{VLpep1}
& & \lambda \phi=d(x)\phi''-g(x)\phi'+h(x)\phi, \quad x\in\mathbb{R},\nonumber\\ 
& &\phi(x+L)=\phi(x),\quad x\in\mathbb{R}
\end{eqnarray}
admits a principal eigenvalue $\lambda(d,g,h)$ associated with a positive $L$-periodic eigenfunction $\phi(x)$(see, e.g., \cite[Theorem 7.6.1]{H} and \cite[Lemma 3.3]{Weng}). 
By \cite[Theorem 2.3.4]{Zhaobook} and similar arguments  to those in \cite[Theorem 3.2]{Weng},  we have the following result.
\begin{proposition}\label{VLexistence}Assume that $L$-periodic functions $d, g,c, e\in C^\nu(\mathbb{R})$, and $d(\cdot)>0, e(\cdot)>0$.  
Let $u(t,x,\phi)$ be the unique solution of the following parabolic equation:
\begin{eqnarray}\label{VLseq}
& & \frac{\partial u}{\partial t}=d(x)\frac{\partial^2 u}{\partial x^2}-g(x)\frac{\partial u}{\partial x}+u(c(x)-e(x)u), \quad t>0,\  x\in\mathbb{R},\nonumber\\
& & u(0,x)=\phi(x)\in Y_+, \quad x\in \mathbb{R}.
\end{eqnarray}  Then the following statements are valid:
\begin{enumerate}
\item[(i)] If $\lambda(d,g,c)\le0$, then $u=0$ is globally asymptotically stable with
respect to initial values in $Y_+$;
\item[(ii)] If $\lambda(d,g,c)>0$, then \eqref{VLseq} admits a unique positive $L$-periodic steady state
$u^*(x)$, and it is globally asymptotically stable with respect to initial values
in $Y_+\backslash\{0\}$.
\end{enumerate}
\end{proposition}  
Let $\mathbb{P}=PC(\mathbb{R},\mathbb{R}^2)$ be the set of all continuous and $L$-periodic functions from $\mathbb{R}$ to $\mathbb{R}^2$, and $\mathbb{P}_+=\{\psi\in \mathbb{P}: \ \psi(x)\ge0,\forall x\in\mathbb{R}\}$. Then $\mathbb{P}_+$ is a closed cone of $\mathbb{P}$ and induces a partial ordering on $\mathbb{P}$. Moreover, we introduce a norm $\|\phi\|_\mathbb{P}$ by
\[\|\phi\|_\mathbb{P}=\max_{x\in \mathbb{R}}|\phi(x)|.\]
It then follows that $(\mathbb{P},\|\phi\|_\mathbb{P})$ is a Banach lattice. 

Clearly, for any $\varphi\in\mathbb{P}$, \eqref{VL} has a unique solution $u(t,\cdot,\varphi)\in\mathbb{P}$ defined on $[0,t_\varphi)$ with $t_\varphi\in(0,\infty]$. By the comparison principle about for scalar reaction-diffusion equations in a period habitat (see, e.g., \cite[Lemma 3.1]{Weng}), together with the fact that $a_{ij}(x)>0, \forall x\in\mathbb{R},1\le i,j\le2$, it follows that for any $\varphi\in\mathbb{P}_+$,  \eqref{VL} has a unique nonnegative solution $u(t,\cdot,\varphi)$ defined on $[0,\infty)$, and $u(t,\cdot,\varphi)\in\mathbb{P}_+$ for all $t\ge0$. 

By Proposition \ref{VLexistence}, we see that there exists two positive $L$-periodic functions $u^*_1(x)$ and $u^*_2(x)$ such that $E_1:=(u^*_1(x),0)$, $E_2:=(0,u^*_2(x))$ are semi-trivial steady states of system \eqref{VL} provided that $\lambda(d_i,g_i,b_i)>0,\ i=1,2.$
Since we mainly concern about the case of the competition exclusion, we impose the following conditions on  system \eqref{VL}:
\begin{enumerate}
\item[(H1)]$\lambda(d_i,g_i,b_i)>0,\ i=1,2.$
\item[(H2)]$\lambda(d_1,g_1,b_1\!-\!a_{12}u^*_2)>0.$
\item[(H3)]System \eqref{VL} has no coexistence steady state, i.e., there is no steady state in Int$(\mathbb{P}_+)$.
\end{enumerate}

(H1) guarantees the existence of two semi-trivial steady states of system \eqref{VL}. (H2) implies that $(0,u^*_2(x))$ is unstable.
Moreover, by Lemma \ref{mp} with $\mu=0, d(x)=d_1(x)$ and $g(x)=g_1(x),\forall x\in\mathbb{R}$, we know that (H2) implies $\lambda(d_1,g_1,b_1)>0$. Thus, we can simply drop the assumption $\lambda_1(d_1,g_1,b_1)>0$ from (H1). 

Under assumptions (H1)--(H3), there are three steady states in $\mathbb{P}_+$: $E_0=(0,0)$, $E_1:=(u^*_1(x),0)$, and $E_2:=(0,u^*_2(x))$. Next, we use the theory developed in \cite{Hsu} for abstract competitive systems (see 
also \cite{Hess2}) to prove the global stability of $E_1$.
\begin{theorem}\label{VLEQ}
Assume that (H1)--(H3) hold. Then $E_1 (u^*_1(x), 0)$ is globally asymptotically stable for all
initial values in $\mathbb{P}_+\backslash\{0,E_2\}$.
\end{theorem}
\begin{proof}
Let $u(t,x,\phi)$ be the solution of system \eqref{VL} with $u(0,x)=\phi(x)$. In view of (H2), we can fix
a real number  $\varepsilon_0\in (0,\lambda(d_1,g_1,b_1\!-\!a_{12}u^*_2))$. By the uniform continuity of $F(x,u):=b_1(x)-a_{11}(x)u_1-a_{12}(x)u_2$ on the set $\mathbb{R}\times[0,1]\times[0,b]$, where $b=\max\limits_{x\in \mathbb{R}}u^*_2(x)+1$,  there exists $\delta_0\in(0,1)$ such that 
\[|F(x,u)-F(x,v)|<\varepsilon_0,\quad \forall u=(u_1,u_2),v=(v_1,v_2)\in[0,1]\times[0,b],\ x\in\mathbb{R}\] provided that $|u_i-v_i|<\delta_0, i=1,2.$ Then we have the following observation.

\smallskip

\noindent {\it Claim.}  $\lim\sup_{t\rightarrow\infty}\|u(t,x,\phi)-(0,u^*_2(x))\|_\mathbb{P}\ge\delta_0$
for any $\phi\in\mathbb{P}_+$ with $\phi_1\not\equiv0$.

  Suppose, by contradiction, that $\lim\sup\limits_{t\rightarrow\infty}\|u(t,x,\hat\phi)-(0,u^*_2(x))\|_\mathbb{P}<\delta_0$ for some $\hat\phi\in\mathbb{P}_+$ with $\hat\phi_1\not\equiv0.$ Then there exists $t_0>0$ such that $$\|u_1(t,x,\hat\phi)\|_{Y}<\delta_0,\  \|u_2(t,x,\hat\phi)-u^*_2(x)\|_Y<\delta_0,\  \forall t\ge t_0.$$
Consequently, we have
 \[F(x, u(t,x,\hat{\phi}))>F(x,(0,u^*_2(x)))-\varepsilon_0=b_1(x)-a_{12}(x)u^*_2(x)-\varepsilon_0,\quad  t\ge t_0,\  x\in \mathbb{R}.\]
 Let $\psi_1(x)$ be a positive eigenfunction corresponding to the principal eigenvalue $\lambda(d_1,g_1,b_1\!-\!a_{12}u^*_2)$. Then $\psi_1(x)$ satisfies
 \begin{align}\label{VLeq1}
 &  \lambda(d_1,g_1,b_1\!-\!a_{12}u^*_2)\psi_1\!=\!d_1(x)\psi_1''\!-\!g_1(x)\psi_1'\!+\!(b_1(x)\!-\!a_{12}(x)u^*_2(x))\psi_1,\quad x\in\mathbb{R}, \nonumber\\ 
 & \psi_1(x+L)=\psi_1(x),\quad x\in\mathbb{R}.
 \end{align}
 Since $u_1(0,x)=\hat{\phi_1}\not\equiv0$, by the comparison principle (see, e.g., \cite[Lemma 3.1]{Weng}), as applied to the first equation in system \eqref{VL}, implies that $u_1(t_0,x,\hat\phi)>0,\ \forall x\in\mathbb{R}$. Then there exists small $\eta>0$ such that $u_1(t_0,\cdot)\ge\eta\psi_1\gg0$. Thus, $u_1(t,x,\hat\phi)$ satisfies
 \begin{eqnarray}\label{VLeq2}
 & & \frac{\partial u_1}{\partial t}\ge L_1u_1+u_1(b_1(x)-a_{12}(x)u^*_2(x)-\varepsilon_0), \quad t>t_0,\ x\in\mathbb{R},\nonumber\\ 
 & &u_1(t_0,\cdot)\ge\eta\psi_1.
 \end{eqnarray}
 In view of \eqref{VLeq1}, it easily follows that $v(t,\cdot)=\eta e^{[\lambda(d_1,g_1,b_1-a_{12}u^*_2)-\varepsilon_0](t-t_0)}\psi_1$ satisfies
  \begin{eqnarray}\label{VLeq3}
  & & \frac{\partial v}{\partial t}= L_1v+v(b_1(x)-a_{12}(x)u^*_2(x)-\varepsilon_0), \quad t>t_0, x\in\mathbb{R},\nonumber\\ 
  & &u_1(t_0,\cdot)=\eta\psi_1.
  \end{eqnarray}
  By \eqref{VLeq2} and \eqref{VLeq3}, together with the standard comparison principle, it follows that  
  \[u_1(t,\cdot,\hat\phi)\ge\eta e^{[\lambda(d_1,g_1,b_1-a_{12}u^*_2)-\varepsilon_0](t-t_0)}\psi_1,\quad \forall  t\ge t_0.\] 
  Letting $t\rightarrow \infty$, we see that $u_1(t,\cdot,\hat{\phi})$ is unbounded, a contradiction.
  
   By the above claim and (H3), we rule out possibility (a) and (c) in \cite[Theorem B]{Hsu}. Since $E_2$ is repellent in some neighborhood of itself, \cite[Theorem B]{Hsu} implies that $E_1$ is globally asymptotically stable
   for all initial values in $\mathbb{P}_+\backslash\{0,E_2\}$.   
\end{proof}
\section{Spreading speeds and traveling waves}
In this section, we study the spreading speeds and spatially periodic traveling waves for system \eqref{VL}. 
By a change of variables $v_1=u_1, v_2=u^*_2(x)-u_2$, we transform system \eqref{VL} into the following cooperative system:
\begin{align}\label{NModel}
&\frac{\partial v_1}{\partial t}\!=\!L_1v_1\!+\!v_1(b_1(x)\!-\!a_{12}(x)u^*_2(x)\!-\!a_{11}(x)v_1\!+\!a_{12}(x)v_2),\quad t>0,\  x\in\mathbb{R},\nonumber \\
&\frac{\partial v_2}{\partial t}\!=\!L_2v_2\!+\!a_{21}(x)v_1(u^*_2(x)\!-\!v_2)
+\!v_2(b_2(x)\!-\!2a_{22}(x)u^*_2(x)\!+\!a_{22}(x)v_2).
\end{align}
Note that three steady states of \eqref{VL}, respectively, become
\[\hat E_0=(0,u^*_2(x)),\ \hat E_1=(u^*_1(x),u^*_2(x)),\ \hat E_2=(0,0).\] 

Let $\mathcal{C}$ be the set of all bounded and continuous functions
from $\mathbb{R}$ to $\mathbb{R}^2$ and $\mathcal{C}_+=\{\phi\in\mathcal{C}:\phi(x)\geq0,\ \forall x\in \mathbb{R}\}$.  Assume that $\beta$ is a  strongly positive $L$-periodic continuous function from $\mathbb{R}$ to $\mathbb{R}^2$. Set 
\begin{small}$$\mathcal{C}_{\beta}=\{u\in \mathcal{C}:\, 0\le u(x)\le \beta(x),\ \forall x\in \mathbb{R}\},\  \mathcal{C}^{per}_{\beta}=\{u\in \mathcal{C_\beta}:\,  u(x)=u(x+L),\ \forall x\in \mathbb{R}\}.$$\end{small} Let $X=C([0,L],\mathbb{R}^2)$ equipped with the maximum norm $|\cdot|_X$, $X_+=C([0,L],\mathbb{R}_+^2)$, $$X_{\beta}=\{u\in X:\ 0\le u(x)\le{\beta}(x),\  \forall x\in[0,L]\},\  \text{and}\  \overline{X}_{\beta}=\{u\in X_{\beta}: u(0)=u(L)\}.$$ Let $BC(\mathbb{R}, X)$ be the set of all continuous and bounded functions from $\mathbb{R}$ to $X$. Then we define \begin{small}$$\mathcal{X}=\{v\in BC(\mathbb{R},X):v(s)(L)=v(s+L)(0),\forall s\in \mathbb{R}\},  \mathcal{X}_+=\{v\in \mathcal{X}:v(s)\in X_+,\forall s\in \mathbb{R}\},$$\end{small} and  
$$\mathcal{X}_{\beta}=\{v\in BC(\mathbb{R},X_{\beta}):v(s)(L)=v(s+L)(0),\forall s\in \mathbb{R}\}.$$ We
equip $\mathcal{C}$ and $\mathcal{X}$ with the compact open topology, that is, $u_n\to
u$ in $\mathcal{C}$ or $\mathcal{X}$ means that the sequence of $u_n(s)$ converges to $u(s)$ in $\mathbb{R}^2$ or $X$ uniformly for $s$ in any compact set. We equip $\mathcal{C}$ and $\mathcal{X}$ with the norm $\|\cdot\|_\mathcal{C}$ and $\|\cdot\|_\mathcal{X}$, respectively, by \[\|u\|_{\mathcal{C}}=\sum\limits_{k=1}^{\infty}\frac{\max_{|x|\le k}|u(x)|}{2^k},\ \forall u\in\mathcal{C},\] 
where $|\cdot|$ denotes the usual norm in $\mathbb{R}^2$, and
\[\|u\|_{\mathcal{X}}=\sum\limits_{k=1}^{\infty}\frac{\max_{|x|\le k}|u(x)|_X}{2^k},\ \forall u\in\mathcal{X}.\] 

    Let $\beta(\cdot)=(u^*_1(\cdot),u^*_2(\cdot))$, and $T_1(t)$ and $T_2(t)$ be the linear semigroups generated by $$\frac{\partial v}{\partial t}\!=\!L_1v+v(b_1(x)-a_{12}(x)u^*_2(x))\ \text{and}\ \frac{\partial v}{\partial t}\!=\!L_2v
    +v(b_2(x)-2a_{22}(x)u^*_2(x)),$$ respectively. It follows that $T_1(t)$ and $T_2(t)$ are compact with the respect to the compact open topology for each $t>0$ (see, e.g., \cite{Weng}). For any $u=(u_1,u_2)\in \mathcal{C}_\beta$, define $F:\mathcal{C}_\beta\rightarrow \mathcal{C}$ by
    $$F(u)=\left(\begin{array}{c}
  -a_{11}(x)u^2_1+a_{12}(x)u_1u_2\\
  a_{21}(x)u^*_2(x)u_1-a_{21}(x)u_1u_2+a_{22}(x)u^2_2
    \end{array}\right).$$ Then we rewrite system \eqref{NModel} as an integral equation form:
    \begin{align}\label{inteq}& v(t)=T(t)v(0)+\int_{0}^{t}T(t-s)F(v(s))ds,\quad t>0,\nonumber\\
    & v(0)=\phi\in\mathcal{C}_\beta,
    \end{align}
    where $T(t)=diag(T_1(t),T_2(t))$. 
    
    As usual, a solution of (\ref{inteq}) is called a mild solution of system \eqref{NModel}. 
    It then follows that for any $\phi\in \mathcal{C}_\beta$, system \eqref{NModel} has a mild solution $u(t,\cdot,\phi)$ defined on $[0,\infty)$ with $u(0,\cdot,\phi)=\phi$, and $u(t,\cdot,\phi)\in\mathcal{C}_\beta$ for all $t\ge0$, and it is a classical solution when $t>0$.
   
   	We say that $V(x-ct,x)$ is an $L$-periodic rightward traveling wave of  system \eqref{NModel}
       	 if $V(\cdot+a,\cdot)\in \mathcal{C}_\beta$,\ $\forall a\in \mathbb{R}$, $u(t,x,V(\cdot,\cdot))=V(x-ct,x)$, $\forall t\ge0$, and $V(\xi,x)$ is an $L$-periodic function in $x$ for any fixed $\xi\in\mathbb{R}$. Moreover, we say that $V(\xi,x)$ connects $\beta$ to $0$ if $\lim_{\xi\rightarrow -\infty}|V(\xi,x)-\beta(x)|=0$ and $\lim_{\xi\rightarrow +\infty}|V(\xi,x)|=0$ uniformly for $x\in\mathbb{R}$.
       	
       	 \begin{definition}
    	A function $u(x,t)$ is said to be an upper (a lower) solution of system \eqref{NModel} if it satisfies
    	$$ 
    	u(t)\geq(\leq) T(t) u(0)+ \int_0^t T(t-s)F(u(s))ds,\quad t\ge0.
    	$$
    	 \end{definition}
    Define a family of operators  $\{Q_t\}_{t\ge0}$ on $\mathcal{C}_{\beta}$ by  $Q_t(\phi):=u(t,\cdot,\phi)$, where $u(t,\cdot,\phi)$ is the solution of system \eqref{NModel} with $u(0,\cdot)=\phi\in\mathcal{C}_{\beta}$. It then easily follows that $\{Q_t\}_{t\ge0}$ is a monotone semiflow on $\mathcal{C}_{\beta}$. Note that if $u(t,x,\phi)$ is a solution of \eqref{NModel}, so is $u(t,x-a,\phi),\ \forall a\in L\mathbb{Z}$. This implies that (A1) in the Appendix holds. By Theorem \ref{VLEQ}, we know that for each $t>0$, $(A4)$ holds for $Q_t$. Since $T(t)$ is compact with the compact open topology for each $t>0$, (A2) and (A5) then follow from the same argument as in \cite[Theorem 8.5.2]{Martin}. Thus, we have the 
    following observation.
    \begin{proposition}
    Assume that (H1)--(H3) hold. Then for each $t>0$, $Q_t$ satisfies assumptions (A1)--(A5) in the Appendix.
    \end{proposition}
With the help of $\{Q_t\}_{t\ge0}$, we can introduce a family of operators $\{\hat Q_t\}_{t\ge0}$ on $\mathcal{X}$:
\begin{equation}
\hat Q_t[v](s)(\theta):=Q_t[v_s](\theta),\quad \forall v\in\mathcal{X},\ s\in\mathbb{R},\ \theta\in[0,L], t\ge0,
\end{equation}
where $v_s\in\mathcal{C}$ is defined by  
$$v_s(x)=v(s+n_x)(\theta_x),\quad \forall x=n_x+\theta_x\in\mathbb{R},\ n_x=L\left[\frac{x}{L}\right],\ \theta_x\in[0,L).$$

By Proposition \ref{HQASS}, it is easy to see that $\{\hat Q_t\}_{t\ge0}$ is a monotone semiflow on $\mathcal{X}_\beta$ and $\hat Q_t$ satisfies (B1)--(B5) in the Appendix for each $t>0$.
Now we follow the procedure in the Appendix with $m=2$. 
Let $c^*_+$ and $\overline c_+$ be defined in \eqref{defc} for the rightward direction of spreading with $\tilde P=\hat{Q}_1$. To show that $\overline{c}_+$ is the minimal wave speed for $L$-periodic traveling waves of system \eqref{NModel} connecting $\beta$ to $0$, we need the following assumption:
\begin{enumerate}
\item[(H4)]$c^*_{1+}+c^*_{2-}>0$, where $c^*_{1+}$ and $c^*_{2-}$ are the rightward and leftward spreading speeds 
of \eqref{u1} and \eqref{u2}, respectively.
\end{enumerate}

 We remark that if $L_iu=\frac{\partial}{\partial x}(d_i(x)\frac{\partial u}{\partial x})$ with $d_i\in C^{1+\nu}(\mathbb{R})$, or all the coefficient functions in \eqref{u1} and \eqref{u2} are  even except $g_i$ is odd, $i=1,2$, Lemma \ref{H45} shows that (H1) and (H2) guarantee (H4).

\begin{theorem}\label{MIN}
Assume that (H1)--(H4) hold. Then for any $c\ge\overline{c}_+$, system \eqref{NModel} admits an L-periodic traveling wave $(U(x-ct,x),V(x-ct,x))$ connecting $\beta$ to $0$, with wave profile components $U(\xi,x)$ and $V(\xi,x)$ being continuous and non-increasing in $\xi$, and for any $c<\overline{c}_+$, there is no such traveling wave connecting $\beta$ to $0$.
\end{theorem} 
\begin{proof}
(i) In view of Theorem \ref{Tw} (2) and (3), it suffices to rule out the second case in Theorem \ref{Tw} (2). Suppose, by contradiction, that the statement in Theorem \ref{Tw} (2)(ii) is valid for some $c\ge\overline{c}_+$. Note that system \eqref{NModel} has exactly three $L$-periodic nonnegative steady states and $\hat E_0=(0,u_2^*(x))$ is the only intermediate equilibrium between $\hat E_1=\beta$ and $\hat E_2=0$, then we have $\alpha_1=\alpha_2=\hat E_0$. Thus, by restricting  system \eqref{NModel} on the order interval $[\hat{E}_0,\hat{E}_1]$ and $[\hat{E}_2,\hat{E}_0]$, respectively, we see that one scalar equation
\begin{equation}\label{u1}
 u_t=L_1u+u(b_1(x)-a_{11}(x)u)
\end{equation}
admits an $L$-periodic traveling wave $U(x-ct,x)$ connecting $u^*_1(x)$ to $0$ with $U(\xi,x)$ being continuous and nonincreasing in $\xi$, and the other scalar equation 
\begin{equation}
v_t=L_2v+v(b_2(x)-2a_{22}(x)u^*_2+a_{22}(x)v)
\end{equation}
also admits an $L$-periodic traveling wave $V(x-ct,x)$ connecting $u_2^*(x)$
to $0$ with $V(\xi,x)$ being continuous and nonincreasing in $\xi$.
 
Let $W(x-ct,x)=u^*_2(x)-V(x-ct,x)$. Then $W(x-ct,x)$ is an $L$-periodic traveling wave connecting $0$ to $u^*_2(x)$  of the following scalar equation with $W(\xi,x)$ being continuous and nondecreasing in $\xi$
\begin{equation}\label{u2}
w_t=L_2w+w(b_2(x)-a_{22}(x)w).
\end{equation}
  Note that $W(x-ct,x)$ is an $L$-periodic leftward traveling wave connecting $0$ to $u^*_2$ with wave speed $-c$, and that systems \eqref{u1} and \eqref{u2} admit rightward spreading speed $c^*_{1+}$ and leftward spreading speed $c^*_{2-}$, respectively, which are also the rightward and the leftward minimal wave speeds (see, e.g., \cite[Theorem 5.3]{Liang2}). It then follows that $c\ge c^*_{1+}$ and  $-c\ge c^*_{2-}$. This implies that $c^*_{1+}+c^*_{2-}\le 0$, a contradiction.
\end{proof}
Let $\lambda_2(\mu)$ be the principle eigenvalue of the elliptic eigenvalue problem: 
\begin{small}\begin{align}\label{eep2}
&\lambda \psi=d_2(x)\psi''\!-\!(2\mu d_2(x)+g_2(x))\psi'{+}\left(d_2(x)\mu^2\!+\!g_2(x)\mu\!+\!b_2(x)\!-\!a_{22}(x)u^*_2(x)\right)\psi,\quad x\in\mathbb{R},\nonumber\\
&\psi(x+L)=\psi(x),\quad x\in\mathbb{R}.\end{align} \end{small} In order to prove that system \eqref{NModel} admits a single rightward spreading speed, we impose the following assumption:
\begin{enumerate}

\item[(H5)] $\limsup_{\mu\to 0^+}\frac{\lambda _2(\mu)}{\mu}\leq c_{1+}^*$, where $c_{1+}^*$ is the rightward spreading speed of \eqref{u1}.
\end{enumerate} 

 By virtue of Lemma \ref{H45}, it follows that in the case where $L_iu=\frac{\partial}{\partial x}(d_i(x)\frac{\partial u}{\partial x})$ with $d_i\in C^{1+\nu}(\mathbb{R})$, or  all the coefficient functions of system \eqref{NModel} are even except $g_i$ is odd, $i=1,2$, (H5) is automatically satisfied provided that (H1) and (H2) hold true.
\begin{theorem}\label{Qspreading}
Assume that (H1)--(H5) hold. Then the following statements are valid for system \eqref{NModel}:
\begin{enumerate}
		\item[(i)]If $\phi\in\mathcal{C}_{\beta}$, $0\le \phi\le \omega\ll \beta$ for some $\omega\in \mathcal{C}^{per}_{\beta}$, and $\phi(x)=0, \forall x\ge H$, for some $H\in \mathbb{R}$, then $\lim_{t\rightarrow\infty,x\ge ct}u(t,x,\phi)=0$ for any $c>\overline{c}_+$.
		\item[(ii)]If $\phi\in\mathcal{C}_{\beta}$ and $\phi(x)\ge \sigma$, $\forall x\le K$, for some $\sigma\in \mathbb{R}^2$ with $\sigma\gg0$ and $K\in\mathbb{R}$, then $\lim_{t\rightarrow\infty,x\le ct}(u(t,x,\phi)-\beta(x))=0$ for any $c<\overline{c}_+$.
	\end{enumerate}
\end{theorem}
\begin{proof}
In view of Theorem \ref{spreading}, it suffices to show that $\overline{c}_+=c^*_+$. If this is not valid, then the definition of $\overline{c}_+$ and $c^*_+$ implies that $\overline{c}_+>c^*_+$. By Theorem \ref{Tw} (1) and (3), it follows that system \eqref{NModel} admits an $L$-periodic traveling wave $(U_1(x-c^*_+t,x),U_2(x-c^*_+t,x))$ connecting $(u^*_1(x),u^*_2(x))$ to $(0,u^*_2(x))$ with $U_i(\xi,x) (i=1,2)$ being continuous and nonincreasing in $\xi$. Therefore, $U_2\equiv u^*_2(x)$, and $U_1(x-c^*_+t,x)$ is an $L$-periodic traveling wave connecting $u^*_1(x)$ to $0$. This implies $c^*_+\ge c^*_{1+}$ where $c^*_{1+}$ is the rightward spreading of \eqref{u1}. By \cite[Theorem 1.1]{Hamel1}, it follows that $c^*_{1+}=\inf_{\mu>0}\frac{\lambda_1(\mu)}{\mu}$, where $\lambda_1(\mu)$ is the principal eigenvalue of the scalar elliptic eigenvalue problem:
\begin{eqnarray}\label{eep}
& & \lambda \psi=d_1(x)\psi''-(2\mu d_1(x)+g_1(x))\psi'+(d_1(x)\mu^2+g_1(x)\mu+b_1(x))\psi,\quad x\in\mathbb{R},\nonumber\\
& &\psi(x+L)=\psi(x),\quad x\in\mathbb{R}.
\end{eqnarray}
 For any given $c_1\in(c^*_+,\overline{c}_+)$, there exists $\mu_1>0$ such that $c_1=\frac{\lambda_1(\mu_1)}{\mu_1}$. Let $\phi^*_1(x)$ be the positive  $L$-periodic
eigenfunction associated with the principal eigenvalue $\lambda_1(\mu_1)$ of \eqref{eep}. Then it easily follows that 
$$u_1(t,x):=e^{-\mu_1(x-c_1t)}\phi^*_1(x)=e^{-\mu_1x}e^{\lambda_1(\mu_1)t}\phi^*_1(x),\quad t\ge0,\ x\in\mathbb{R},$$
is a solution of the linear equation
$$
\frac{\partial u_1}{\partial t}=L_1u_1+b_1(x)u_1.
$$
Since $c^*_{1+}<c_1$ and (H5) holds,  we can choose a
small number $\mu_2\in (0, \mu_1)$ such that  $c_2:=\frac{\lambda_2(\mu_2)}{\mu_2}< c_1$.
  Let $\phi_2^*(x)$ be the positive $L$-periodic
eigenfunction associated with the principal eigenvalue $\lambda_2(\mu_2)$ of \eqref{eep2}.
It is easy to see that
$$
u_2(t,x):=e^{-\mu_2(x-c_2t)}\phi^*_2(x)=e^{-\mu_2 x} e^{\lambda _2(\mu_2)t}\phi^*_2(x)
$$
is a solution of the linear equation
\begin{equation}\label{Eq1}
\frac{\partial u_2}{\partial t}=L_2u_2+(b_2(x)-a_{22}(x)u^*_2(x))u_2.
\end{equation}
Since $c_1>c_2$, it follows that the function
$$
v_2(t,x):=e^{-\mu_2(x-c_1t)}\phi^*_2(x)=
e^{\mu_2(c_1-c_2)t}u_2(t,x),\quad t\ge0,\ x\in\mathbb{R},
$$
satisfies
\begin{equation}\label{ineq}
\frac{\partial v_2}{\partial t}\ge L_2v_2+(b_2(x)-a_{22}(x)u^*_2(x))v_2.
\end{equation}

Define two wave-like functions: 
\begin{equation}
\overline{u}_1(t,x):=\min\{m_0e^{-\mu_1(x-c_1t)}\phi^*_1(x),u^*_1(x)\},\quad t\ge0,\ x\in \mathbb{R},
\end{equation} and
\begin{equation}
\overline{u}_2(t,x):=\min\{q_0e^{-\mu_2(x-c_1t)}\phi^*_2(x),u^*_2(x)\},\quad t\ge0,\ x\in\mathbb{R},
\end{equation}
where
$$q_0:= \max_{x\in[0,L]}\frac{u^*_2(x)}{\phi^*_2(x)}>0,\quad m_0:=\min_{x\in[0,L]}\frac{q_0a_{22}(x)\phi^*_2(x)}{a_{21}(x)\phi^*_1(x)}>0.$$
Now, we are ready to verify that $(\overline{u}_1,\overline{u}_2)$ is an upper solution to system \eqref{NModel}. Indeed, for all $x-c_1 t>\frac{1}{\mu_1}\ln\frac{m_0\phi_1^*(x)}{u^*_1(x)}$, we have $\overline{u}_1(t,x)=m_0e^{-\mu_1(x-c_1t)}\phi^*_1(x)$, and hence,
\begin{eqnarray}
& &\frac{\partial \overline{u}_1}{\partial t}-L_1\overline{u}_1-\overline{u}_1(b_1(x)-a_{12}(x)u^*_2(x)-a_{11}(x)\overline{u}_1+a_{12}(x)\overline{u}_2)\nonumber\\
& &\ge\frac{\partial \overline{u}_1}{\partial t}-L_1\overline{u}_1-b_1(x)\overline{u}_1=0\nonumber.
\end{eqnarray}
For all $x-c_1t<\frac{1}{\mu_1}\frac{m_0\phi_1^*(x)}{u^*_1(x)}$, we obtain $\overline{u}_1(t,x)=u^*_1(x)$, and hence,
\begin{eqnarray*}
& &\frac{\partial \overline{u}_1}{\partial t}-L_1\overline{u}_1-\overline{u}_1(b_1(x)-a_{12}(x)u^*_2(x)-a_{11}(x)\overline{u}_1+a_{12}(x)\overline{u}_2)\\
& &\ge\frac{\partial \overline{u}_1}{\partial t}-L_1\overline{u}_1-\overline{u}_1(b_1(x)-a_{11}(x)\overline{u}_1)=0.
\end{eqnarray*}
On the other hand, for all $x-c_1t\!>\!\frac{1}{\mu_2}\ln\frac{q_0\phi_2^*(x)}{u^*_2(x)}\!>\!0$, it follows that  $$\overline{u}_2(t,x)=q_0e^{-\mu_2(x-c_1t)}\phi^*_2(x),$$ which satisfies inequality  \eqref{ineq}. Note that $$\overline{u}_1(t,x)\le m_0e^{-\mu_1(x-c_1t)}\phi^*_1(x), \quad \forall t\ge0,\ x\in\mathbb{R},$$ and $\mu_2\in(0,\mu_1)$, we get  
$$
\begin{array}{l}
\frac{\partial \overline{u}_2}{\partial t}-L_2\overline{u}_2-a_{21}(x)(u^*_2(x)-\overline{u}_2)\overline{u}_1-\overline{u}_2(b_2(x)-2a_{22}(x)u^*_2(x)+a_{22}(x)\overline{u}_2)\nonumber\\
=\frac{\partial \overline{u}_2}{\partial t}-L_2\overline{u}_2-(b_2(x)-a_{22}(x)u^*_2(x))\overline{u}_2+(u^*_2(x)-\overline{u}_2)(a_{22}(x)\overline{u}_2-a_{21}(x)\overline{u}_1)\\
\ge(u^*_2(x)-\overline{u}_2)e^{-\mu_1(x-c_1t)}a_{21}(x)\phi^*_1(x)(\frac{q_0a_{22}(x)\phi^*_2(x)}{a_{21}(x)\phi^*_1(x)}-m_0)\\
\ge0.
\end{array}
$$
For all $x-c_1t<\frac{1}{\mu_2}\ln\frac{q_0\phi_2^*(x)}{u^*_2(x)}$, we have $\overline{u}_2(t,x)=u^*_2(x)$. Therefore,
\begin{align*}
& \frac{\partial \overline{u}_2}{\partial t}-L_2\overline{u}_2-a_{21}(x)(u^*_2(x)-\overline{u}_2)\overline{u}_1-\overline{u}_2(b_2(x)-2a_{22}(x)u^*_2(x)+a_{22}(x)\overline{u}_2)\\
&=-L_2u^*_2-u^*_2(b_2(x)-a_{22}(x)u^*_2)=0.
\end{align*}
It then follows that $\overline{u}=(\overline{u}_1,\overline{u}_2)$ is a continuous upper  solution of system \eqref{NModel}.
\smallskip

  Let $\phi\in \mathcal{C}_\beta$  with $\phi(x)\ge \sigma$, $\forall x\le K$ and $\phi(x)=0$, $\forall x\ge H$, for some $\sigma\in \mathbb{R}^2$ with $\sigma\gg0$ and $K,H\in\mathbb{R}$. By the arguments in \cite[Lemma 2.2]{WLL2002} and the proof of Theorem \ref{spreading}, as applied to $\hat Q_1$,  it follows that for any $c<\overline{c}_+$, there exists $\delta(c)>0$ such that 
  \begin{equation}\label{ineq2}{\liminf}_{n\rightarrow\infty, x\le cn}|u(n,x,\phi)|\ge\delta(c)>0.\end{equation} Moreover, there exists a sufficiently large positive constant $A\in L\mathbb{Z}$ such that 
$$\phi(x)\le\overline{u}(0,x-A) :=\psi(x),\quad \forall x\in \mathbb{R}.$$
By the translation invariance of $Q_t$, it follows that $\overline{u}(t,x-A)$ is still an upper solution of system \eqref{NModel}, and hence,
\begin{equation}\label{ineq3}0\le u(t,x,\phi)\le u(t,x,\psi)=\overline{u}(t,x-A),\quad \forall x\in\mathbb{R},\  t\ge0.\end{equation}
Fix a number $\hat c\in(c_1,\overline{c}_+)$. Letting $t=n$, $x=\hat cn$ and $n\rightarrow\infty$ in \eqref{ineq3}, together with \eqref{ineq2}, we have 
$$0<\delta(\hat c)\le\liminf_{n\rightarrow\infty}|u(n,\hat{c}n,\phi)|\le\lim_{n\rightarrow\infty}|\overline{u}(n,\hat cn-A)|=0,$$  
which is a contradiction. Thus, $c^*_+=\overline{c}_+$.
\end{proof}

Note that the leftward case can be addressed in a similar way. Indeed, by making a change of variable $v(t,x)=u(t,-x)$ for system \eqref{NModel}, we obtain similar results for the rightward case of the resulting system, which is the leftward case for system \eqref{NModel}. 
\begin{remark}\label{Re1}
In the case where $L_iu=\frac{\partial}{\partial x}(d_i(x)\frac{\partial u}{\partial x})$ with $d_i\in C^{1+\nu}(\mathbb{R})$ in system \eqref{NModel}, $i=1,2$, or all the coefficient functions of system \eqref{NModel} are even except $g_i$ is odd, $i=1,2$, it follows from Lemma  \ref{H45} that system \eqref{NModel} admits a single rightward spreading speed which is coincident with the minimal rightward wave speed provided that (H1)--(H3) hold.    
\end{remark}
\section{Linear determinacy of spreading speed}
In this section, we give a set of sufficient conditions for the rightward spreading speed to be determined by the linearization of system \eqref{NModel} at $\hat E_2=(0,0)$, which is
\begin{align}\label{linear}
&\frac{\partial v_1}{\partial t}\!=\!L_1v_1+(b_1(x)-a_{12}(x)u^*_2(x))v_1, \\
&\frac{\partial v_2}{\partial t}\!=\!L_2v_2\!+\!a_{21}(x)u^*_2(x)v_1
+\!(b_2(x)\!-\!2a_{22}(x)u^*_2(x))v_2,\quad t>0,\  x\in\mathbb{R}\nonumber.
\end{align} 

Clearly, under (H2) the following scalar equation 
\begin{equation}\label{c0eq}
\frac{\partial u}{\partial t}= L_1u+u(b_1(x)-a_{12}(x)u^*_2(x)-a_{11}(x)u),\quad t>0, x\in\mathbb{R}, \\
\end{equation}
admits a rightward spreading speed (also the minimal rightward wave speed) $c^0_+=\inf\limits_{\mu>0}\frac{\lambda_0(\mu)}{\mu}$ (see, e.g., \cite[Theorem 1.1]{Hamel1}), where $\lambda_0(\mu)$ is the principle eigenvalue of the following elliptic eigenvalue problem:
\begin{small}
\begin{align}\label{eep0}
&  \lambda \psi\!=\!d_1(x)\psi''\!-\!(2\mu d_1(x)+g_1(x))\psi'\!+\!(d_1(x)\mu^2\!+\!g_1(x)\mu\!+\!b_1(x)\!-\!a_{12}(x)u^*_2(x))\psi,\quad x\in\mathbb{R},\nonumber\\
& \psi(x+L)=\psi(x),\quad x\in\mathbb{R}.
\end{align}
\end{small}
The next result shows that $c^0_+$ is a lower bound of the slowest spreading $c^*_+$ of system \eqref{NModel}.
\begin{proposition}\label{lb}Let (H1)--(H3) hold. Then $c^*_+\ge c^0_+$.
\end{proposition}
\begin{proof}In the case where $\overline{c}_+>c^*_+$, by the same arguments as in Theorem \ref{Qspreading}, we see that $c^*_+\ge c^*_{1+}$ where $c^*_{1+}$ is the rightward spreading speed of \eqref{u1}. Since $b_1(x)>b_1(x)-a_{12}(x)u^*_2(x), \forall x\in \mathbb{R}$, by Lemma \ref{mp} with $d(x)=d_1(x)$ and $g(x)=g_1(x),\forall x\in\mathbb{R}$, it is easy to see that $\lambda_1(\mu)>\lambda_0(\mu), \forall \mu\ge0$,  where  $\lambda_1(\mu)$ is the principal eigenvalue of \eqref{eep}. Thus, we have $c^*_+\ge c^*_{1+}>c^0_+$.

In the case where $\overline{c}_+=c^*_+$, let $u(t,\cdot,\phi)=(u_1(t,\cdot,\phi),u_2(t,\cdot,\phi))$ be the solution of system \eqref{NModel} with $u(0,\cdot)=\phi=(\phi_1,\phi_2)\in \mathcal{C}_\beta$. Then the positivity of the solution implies that
$$\frac{\partial u_1}{\partial t}\ge L_1u_1+u_1(b_1(x)-a_{12}(x)u^*_2(x)-a_{11}(x)u_1),\quad t>0, x\in\mathbb{R}.$$
Let $v(t,x,\phi_1)$ be the unique solution of \eqref{c0eq} with $v(0,\cdot)=\phi_1$. Then the comparison principle yields that
 \begin{equation}\label{ineq1}u_1(t,x,\phi)\ge v(t,x,\phi_1), \quad \forall t\ge0,\  x\in \mathbb{R}.\end{equation}
 Since $\lambda(d_1,g_1,b_1-a_{12}u^*_2)>0$, Proposition \ref{VLexistence} implies that there exists a unique positive $L$-periodic steady state $v_0(x)$ of \eqref{c0eq}. Let $\phi^0=(\phi_1^0,\phi_2^0)\in \mathcal{C}_\beta$ 
 be chosen as in Theorem \ref{Qspreading} (i) and (ii) such that  $\phi_1^0\leq  v_0$.
 Assume, by contradiction, that $c^*_+< c^0_+$. Since $\overline{c}_+=c^*_+$, we can fix a real number
  $\hat c\in(\overline{c}_+,c^0_+)$. Then Theorem  \ref{Qspreading} implies that 
  $\lim_{t\rightarrow\infty,x\ge \hat ct}u_1(t,x,\phi^0)=0$. By Theorem \ref{spreading}, as applied to system \eqref{c0eq}, we further obtain $\lim_{t\rightarrow\infty,x\le \hat ct}(v(t,x,\phi^0_1)-v_0(x))=0$.  
However, letting $x=\hat c t$ in \eqref{ineq1},  we get  $\lim_{t\rightarrow\infty,x= \hat ct}(v(t,x,\phi^0_1))=0$, which is a contradiction.
\end{proof}

For any given $\mu\in \mathbb{R}$, letting $v(t,x)=e^{-\mu x}u(t,x)$ in \eqref{linear}, we then have 
\begin{eqnarray}\label{uModel}
\frac{\partial u_1}{\partial t}&\!=&\!L_1u_1\!-\!2\mu d_1(x)\frac{\partial u_1}{\partial x}\!+\!(d_1(x)\mu^2\!+\!g_1(x)\mu\!+\!b_1(x)\!-\!a_{12}(x)u^*_2(x))u_1, \nonumber\\ 
\frac{\partial u_2}{\partial t}&\!=&\!L_2u_2\!-\!2\mu d_2(x)\frac{\partial u_2}{\partial x}+a_{21}(x)u^*_2(x)u_1\\
& &+(d_2(x)\mu^2\!+\!g_2(x)\mu\!+\!b_2(x)\!-\!2a_{22}(x)u^*_2(x))u_2,\quad t>0, x\in\mathbb{R}\nonumber.
\end{eqnarray}
Substituting $u(t,x)=e^{\lambda t}\phi(x)$ into \eqref{uModel}, we obtain the following periodic eigenvalue problem:
\begin{align}\label{Lpep}
\lambda \phi_1&= d_1(x)\phi_1''\!-\!(2\mu d_1(x)+g_1(x))\phi_1'+(d_1(x)\mu^2\!+\!g_1(x)\mu\!+\!b_1(x)\!-\!a_{12}(x)u^*_2(x))\phi_1,\nonumber \\
\lambda \phi_2&=d_2(x)\phi_2''\!-\!(2\mu d_2(x)+g_2(x))\phi_2'+a_{21}(x)u^*_2(x)\phi_1\nonumber\\
&  +\left(d_2(x)\mu^2\!+\!g_2(x)\mu\!+\!b_2(x)\!-\!2a_{22}(x)u^*_2(x)\right)\phi_2,\quad  x\in\mathbb{R},\\
\phi_i(x)&=\phi_i(x+L), \quad \forall x\in\mathbb{R},\ i=1,2.\nonumber
\end{align}
Let $\overline{\lambda}(\mu)$ be the principal eigenvalue of the following  periodic eigenvalue problem:
\begin{align}\label{Lpep2}
\lambda \psi\!=& d_2(x)\psi''\!-\!(2\mu d_2(x)\!+\!g_2(x))\psi'\nonumber\\
&+\!\left(d_2(x)\mu^2\!+\!g_2(x)\mu\!+\!b_2(x)\!-\!2a_{22}(x)u^*_2(x)\right)\psi,\quad  x\in\mathbb{R},\\
\psi(x)\!=&\psi(x+L), \quad x\in\mathbb{R}.\nonumber
\end{align}
Then there exists $\mu_0>0$ such that $c^0_+=\frac{\lambda_0(\mu_0)}{\mu_0}$. Now we introduce the following condition:
\begin{enumerate}
\item[(D1)]  $\lambda_0 (\mu_0)>\overline{\lambda}(\mu_0)$.
\end{enumerate}

\begin{proposition}\label{ef}
Let (H1)--(H3) and (D1) hold. Then the periodic eigenvalue problem \eqref{Lpep} with $\mu=\mu_0$ has a simple eigenvalue $\lambda_0(\mu_0)$ associated with a positive $L$-periodic eigenfunction $\phi^*=(\phi_1^*,\phi_2^*)$.  
\end{proposition}
\begin{proof}
Clearly, there exists an $L$-periodic eigenfunction $\phi_1^*\gg 0$ associated with the principle eigenvalue $\lambda_0(\mu_0)$ of \eqref{c0eq}. Since the first equation of \eqref{Lpep} is decoupled from the second one, it suffices to show that $\lambda_0(\mu_0)$ has a positive eigenfunction $\phi^*=(\phi^*_1,\phi^*_2)$ in \eqref{Lpep}, where $\phi^*_2$ is to be determined. Let $U(t)$ be the solution semigroup generated by the following linear scalar partial differential equation:
\begin{align*}
&\frac{\partial u}{\partial t}=\!L_2u\!-\!2\mu_0 d_2(x)\frac{\partial u}{\partial x}\!+\!(d_2(x)\mu_0^2\!+\!g_2(x)\mu_0\!+\!b_2(x)\!-\!2a_{22}(x)u^*_2(x))u,\quad t>0, x\in\mathbb{R},\\
& u(0,\cdot)=\varphi\in Y.
\end{align*}
It is easy to see that $U(t)$ is a positive and compact semigroup on $Y$ with its generator 
$$
A=L_2-2\mu_0 d_2(x)\frac{\partial}{\partial x}+(d_2(x)\mu_0^2\!+\!g_2(x)\mu_0\!+\!b_2(x)\!-\!2a_{22}(x)u^*_2(x)).
$$ 
By \cite[Theorem 3.12]{Thieme}, $A$ is resolvent-positive and  
$$(\lambda I-A)^{-1}\phi=\int^{\infty}_{0}e^{-\lambda t}U(t)\phi dt, \quad \forall\lambda>s(A), \phi\in Y,$$ where $s(A)$ is the spectral bound of $A$.
Note that $\overline\lambda(\mu_0)$ is the principal eigenvalue of \eqref{Lpep2}, that is, $s(A)=\overline\lambda(\mu_0)$. Since $\lambda_0(\mu_0)>\overline\lambda(\mu_0)=s(A)$, we can define 
$\phi^*_2=(\lambda_0(\mu_0) I-A)^{-1}a_{21}u^*_2\phi^*_1\gg0$. It then follows that $(\phi^*_1,\phi^*_2)$ satisfies \eqref{Lpep} with $\mu=\mu_0$. Since $\lambda_0(\mu_0)$ is a simple eigenvalue for \eqref{c0eq}, we see that so is $\lambda_0(\mu_0)$ for \eqref{Lpep}.
\end{proof}
From Proposition \ref{ef}, it is easy to see that for any given $M>0$, the function \begin{equation}\label{eqU}
U(t,x)=Me^{-\mu_0x}e^{\lambda_0(\mu_0)t}\phi^*(x),\quad t\ge0,\ x\in\mathbb{R},
\end{equation} is a positive solution of system \eqref{linear}. In order to obtain an explicit formula for the spreading speeding $\overline{c}_+$, we need the following additional condition:
\begin{enumerate}
\item[(D2)] $\frac{\phi^*_1(x)}{\phi^*_2(x)}\ge\max\left\{\frac{a_{12}(x)}{a_{11}(x)},\frac{a_{22}(x)}{a_{21}(x)}\right\},\quad \forall x\in\mathbb{R}$.
\end{enumerate}
Now we are in a position to show that system \eqref{NModel} admits a single rightward spreading speed $\overline{c}_+$, which is linearly determinate.

\begin{theorem}\label{c0}
Let (H1)--(H3) and (D1)--(D2) hold. Then $\overline{c}_+=c^*_+=c^0_+=\inf_{\mu>0}\frac{\lambda_0(\mu)}{\mu}$.
\end{theorem}
\begin{proof}
 First, we verify that $U(t,x)$, as defined in \eqref{eqU}, is an upper solution of system \eqref{NModel}. Since $\frac{U_1}{U_2}=\frac{\phi^*_1}{\phi^*_2}$ and (D2) holds true, it follows that 
\begin{eqnarray}
& &\frac{\partial U_1}{\partial t}\!-\!L_1 U_1-U_1(b_1(x)-a_{12}(x)u^*_2(x)-a_{11}(x)U_1+a_{12}(x)U_2)\nonumber\\
& &=a_{11}(x)U_1U_2\left(\frac{U_1}{U_2}-\frac{a_{12}(x)}{a_{11}(x)}\right)\nonumber\\
& &=a_{11}(x)U_1U_2\left(\frac{\phi^*_1(x)}{\phi^*_2(x)}-\frac{a_{12}(x)}{a_{11}(x)}\right)\ge 0,
\end{eqnarray}
and
\begin{eqnarray}
& &\frac{\partial U_2}{\partial t}\!-\!L_2U_2\!-\!a_{21}(x)U_1(u^*_2(x)\!-\!U_2)
-\!U_2(b_2(x)\!-\!2a_{22}(x)u^*_2(x)\!+\!a_{22}(x)U_2).\nonumber\\
& &=a_{21}(x)U_2^{^2}\left(\frac{U_1}{U_2}-\frac{a_{22}(x)}{a_{21}(x)}\right)\nonumber\\
& &=a_{21}(x)U_2^{^2}\left(\frac{\phi^*_1(x)}{\phi^*_2(x)}-\frac{a_{22}(x)}{a_{21}(x)}\right)\ge0.
\end{eqnarray}
Thus, $U(t,x)$ is an upper solution of \eqref{NModel}. Choose some $\phi^0\in\mathcal{C}_\beta$ satisfying the conditions in Theorem \ref{Qspreading} (i) and (ii). Then there exists a sufficiently large number $M_0>0$ such that 
$$0\le\phi^0(x)\le M_0e^{-\mu_0x}\phi^*(x)=U(0,x),\quad \forall x\in\mathbb{R}.$$
Let $W(t,x)$ be the unique solution of system \eqref{NModel} with $W(0,\cdot)=\phi_0$. Then the comparison principle, together with the fact that $c^0_+\mu_0=\lambda_0(\mu_0)$, leads that  
$$0\!\le\! W(t,x)\!\le\! U(t,x)\!=\!M_0e^{-\mu_0x}e^{\lambda_0(\mu_0)t}\phi^*(x)\!=\!M_0e^{-\mu_0(x-c^0_+)t}\phi^*(x),\quad \forall t\ge0,\ x\in\mathbb{R}.$$
It follows that for any given $\epsilon>0$, there holds
$$0\le U(t,x)\le M_0e^{-\mu_0\epsilon t}\phi^*(x),\quad \forall t\ge0,\ x\ge(c^0_++\epsilon)t,$$
and hence,
$$\lim_{t\rightarrow\infty,x\ge (c^0_++\epsilon)t}U(t,x)=0.$$
By Theorem \ref{Qspreading} (ii), we obtain $c^*_+\le c^0_++\epsilon$. Letting $\epsilon\rightarrow 0$, we have $c^*_+\le c^0_+$.  Assume, by contradiction, that  $\overline{c}_+>c^*_+$. Then the proof of Proposition \ref{lb} shows that $c^*_+>c^0_+$, a contradiction. This implies that $\overline{c}_+=c^*_+$.  In view of Proposition \ref{lb},
it follows that $\overline{c}_+=c^*_+=c^0_+$. 
\end{proof}
To finish this section, we consider the following classical Lotka-Volterra competition model:
 \begin{eqnarray}\label{Eg}
   & &\frac{\partial u_1}{\partial t}=d_1\Delta u_1+r_1u_1(1-u_1-a_1u_2),\quad t>0, \ x\in\mathbb{R}, \\ 
   & &\frac{\partial u_2}{\partial t}=d_2\Delta u_2+r_2u_2(1-a_2u_1-u_2),\quad t>0, \ x\in\mathbb{R}, \nonumber
   \end{eqnarray}
   \noindent where all parameters are positive constants. This system was investigated in \cite{Lewis}. By straightforward computations (see, e.g., \cite{Lewis}), it follows that if $a_1<1$, there are only three constant steady states $E_0=(0,0)$, $E_1=(1,0)$ and $E_2=(0,1)$, thus (H3) is valid. Since $\lambda(d_2,0,r_2)=r_2>0$ and $\lambda(d_1,0,r_1(1-a_1))=r_1(1-a_1)>0$, we see that (H1) and (H2) are also valid. Moreover, (H4) and (H5) are automatically satisfied due to Lemma \ref{H45}. Thus, system \eqref{Eg} admits a single spreading speed $\overline{c}_+$ no matter whether it is linearly determinate.

   Next, we find some conditions under which (D1)--(D2) hold for system \eqref{Eg}.
   By substituting $d_i(x)=d_i$, $b_i(x)=r_i$, $a_{ii}(x)=r_i, i=1,2$,  $a_{12}(x)=r_1a_1$, and $a_{21}(x)=r_2a_2$ into system \eqref{NModel}, we can reduce the eigenvalue problems \eqref{eep0} and \eqref{Lpep2} to
   \begin{eqnarray}\label{neep0}
   & & \lambda \psi\!=\!d_1\psi''\!-\!2\mu d_1\psi'\!+\!(d_1\mu^2\!+\!r_1\!-\!r_1a_1)\psi,\quad x\in\mathbb{R},\nonumber\\
   & &\psi(x+L)=\psi(x),\quad x\in\mathbb{R},
   \end{eqnarray}
   and 
   \begin{align}\label{nLpep2}
   & \lambda \psi=d_2\psi''\!-\!2\mu d_2\psi'+\left(d_2\mu^2\!-r_2\right)\psi,\quad  x\in\mathbb{R},\nonumber\\
   &  \psi(x)=\psi(x+L), \quad x\in\mathbb{R}.
   \end{align}
Then it is easy to see that  two principle eigenvalues 
$$\lambda_0(\mu)=d_1\mu^2+r_1-r_1a_1,\  \overline{\lambda}(\mu)=d_2\mu^2-r_2$$  have positive constant eigenfunctions. 
By virtue of  
$$
c^0_+=\inf\limits_{\mu>0}\frac{\lambda_0(\mu)}{\mu}=\min\limits_{\mu>0}\left\{d_1\mu+\frac{r_1(1-a_1)}{\mu}\right\},
$$ 
it follows that $$c^0_+=2\sqrt{d_1r_1(1-a_1)},\ \mu_0=\sqrt{\frac{r_1(1-a_1)}{d_1}}.$$
Thus,  (D1) is equivalent to
 $$\lambda_0(\mu_0)=2r_1(1-a_1)>\frac{d_2r_1(1-a_1)}{d_1}-r_2=\overline{\lambda}(\mu_0).$$
On the other hand, the eigenvalue problem \eqref{Lpep} can  be simplified as 
\begin{eqnarray}\label{nLpep}
& &\lambda \phi_1=d_1\phi_1''-2\mu d_1\phi_1'+(d_1\mu^2+r_1-r_1a_1)\phi_1,\nonumber \\
& &\lambda \phi_2=d_2\phi_2''\!-\!2\mu d_2\phi_2'+a_2r_2\phi_1 +(d_2\mu^2-r_2)\phi_2,\quad  x\in\mathbb{R},\\
& &\phi_i(x)=\phi_i(x+L), \quad \forall x\in\mathbb{R},\ i=1,2.\nonumber
\end{eqnarray}
Substituting $(\phi^*_1,\phi^*_2)=(1,k)$ into the second equation of \eqref{nLpep}, we get 
$$k=\frac{a_2}{(1-a_1)\frac{r_1}{r_2}(2-\frac{d_2}{d_1})+1}>0.$$
It then follows that (D2) is equivalent to
$$\frac{\phi^*_1}{\phi^*_2}=\frac{(1-a_1)\frac{r_1}{r_2}(2-\frac{d_2}{d_1})+1}{a_2}\ge \max\left\{a_1,\frac{1}{a_2}\right\},$$ 
and hence,
\begin{eqnarray*}
& &(1-a_1)\frac{r_1}{r_2}(2-\frac{d_2}{d_1})+1\ge a_1 a_2 ,\\
& &(1-a_1)\frac{r_1}{r_2}(2-\frac{d_2}{d_1})\ge 0,
\end{eqnarray*}that is, 
   \begin{eqnarray}\label{cd}
   & &\frac{d_2}{d_1}\le2,\nonumber\\
   & &\frac{a_1a_2-1}{1-a_1}\le\frac{r_1}{r_2}(2-\frac{d_2}{d_1}), 
   \end{eqnarray}which also guarantees that (D1) holds.  Thus, under condition \eqref{cd}, we have $\overline{c}_+=c^0_+=2\sqrt{d_1r_1(1-a_1)}$. This result is consistent with  \cite[Theorem 2.1]{Lewis}. 
\begin{remark}
Consider a more general reaction-diffusion competition system in a periodic habitat, that is, 
\begin{eqnarray}\label{GVL}
   & &\frac{\partial u_1}{\partial t}=L_1u_1+u_1f_1(x,u_1,u_2),\\ 
   & &\frac{\partial u_2}{\partial t}=L_2u_2+u_2f_2(x,u_1,u_2),\quad t>0, \ x\in\mathbb{R}, \nonumber
   \end{eqnarray}
where the operator $L_i:=a^{(i)}_2(x)\frac{\partial^2}{\partial x^2}+a^{(i)}_1(x)\frac{\partial}{\partial x}$ with $a^{(i)}_2(x)>0,\forall x\in \mathbb{R}$, i.e., $L_i$ is uniformly elliptic, $i=1,2$. Assume that $a^{(i)}_j(x)$ and $f_i(x,u_1,u_2)$ are periodic in $x$ with the same period and H\"{o}lder continuous in $x$ of order $\nu\in(0,1)$, $1\leq i,j\leq 2$, and $f_i(x,u_1,u_2)$ are differentiable with respect to $u_1$ and $u_2$, $i=1,2$. Moreover, $\partial_{u_1}f_1(x,u_1,0)<0$ and   $\partial_{u_2}f_2(x,0,u_2)<0$, $\forall x\in\mathbb{R}$, and there exists $M_1>0$ and $M_2>0$ such that $f_1(x,M_1,0)\le0$, $f_2(x,0,M_2)\le 0$, $\partial_{u_2}f_1(x,u_1,u_2)<0$ and $\partial_{u_1}f_2(x,u_1,u_2)<0$ for all $(x,u_1,u_2)\in\mathbb{R}\times[0,M_1]\times[0,M_2]$. Then we can obtain analogous results on traveling waves and spreading speeds under similar assumptions to (H1)--(H5)
and (D1)--(D2).
\end{remark}
\section{An application}
In this section, we study the spatially periodic version of a well-known reaction diffusion model \cite{DHMP,LN}:
 \begin{eqnarray}\label{dhmp}
   & &\frac{\partial u_1}{\partial t}=d_1 \Delta u_1+u_1(a(x)-u_1-cu_2), \\ 
   & &\frac{\partial u_2}{\partial t}=d_2\Delta u_2+u_2(a(x)-u_1-u_2),\quad t>0, \ x\in\mathbb{R}, \nonumber
   \end{eqnarray}
where $0<d_1<d_2$, $0\le c\le1$ and $a(x)$ is an $L$-periodic continuous function for some $L>0$. Note that model \eqref{dhmp} with $c=1$ was proposed in \cite{DHMP}.

For convenience, we use the same notations as in sections 2 and 3.  We first present some results on the principle eigenvalue $\lambda_m(\mu)$ of \eqref{geep}.
\begin{lemma}\label{mp}
Assume that $L$-periodic functions $d,g,m\in C^\nu(\mathbb{R})(\nu\in(0,1))$. Let $\lambda_m(\mu)(\mu\in\mathbb{R})$ be the principle eigenvalue of the following elliptic eigenvalue problem:
 \begin{eqnarray}\label{geep}
 & & \lambda \psi=d(x)\psi''-(2\mu d(x)+g(x))\psi'+(d(x)\mu^2+g(x)\mu+m(x))\psi,\quad x\in\mathbb{R},\nonumber\\
 & &\psi(x+L)=\psi(x),\quad x\in\mathbb{R}.
 \end{eqnarray}
Then the following statements are valid:
\begin{enumerate}
\item[(a)]If $m_1(x)\ge m_2(x)$ with $m_1(x)\not\equiv m_2(x),\forall x\in\mathbb{R}$, then $\lambda_{m_1}(\mu)>\lambda_{m_2}(\mu)$, $\forall\mu\in\mathbb{R}$.
\item[(b)] $\lambda_m(\mu)$ is a convex function of $\mu$ on $\mathbb{R}$.
\item[(c)]If either $d,m$ are even and $g$ is odd, or $d\in C^{1+\nu}(R)(\nu\in(0,1))$ and $g(x)=-d'_1(x),\forall x\in\mathbb{R}$, then $\lambda_m(\mu)=\lambda_m(-\mu), \forall \mu\in\mathbb{R}$.
\end{enumerate}
\end{lemma}
\begin{proof}
By similar arguments to those in \cite[Lemma 15.5]{Hess} 
, it is easy to prove that (a) holds. (b) follows from the same arguments as in  \cite[Proposition 4.1]{Weng}.

In the case where $d,m$  are even functions and $g$ is odd, for any given $\mu\in\mathbb{R}$, let $\psi(x)$ be the positive and $L$-periodic eigenfunction
associated with $\lambda_m (\mu)$. Then we have
\begin{align}
\lambda_m (\mu)\psi(-x)=& d(-x)\psi''(-x)-(2\mu d(-x)+g(-x))\psi'(-x)\\
&+(d(-x)\mu^2+g(-x)\mu+m(-x))\psi(-x),\quad\forall x\in\mathbb{R}.\nonumber
\end{align}
Letting $\varphi(x)=\psi(-x), x\in\mathbb{R}$, we then have $\varphi'(x)=-\psi'(-x),\varphi''(x)=\psi''(-x), \forall x\in\mathbb{R}$. Since $d(x)=d(-x), m(x)=m(-x), g(x)=-g(-x), \forall x\in\mathbb{R}$,
we obtain    
$$
\lambda_m(\mu) \varphi=d(x)\varphi''+(2\mu d(x)-g(x))\varphi'+(d(x)\mu^2-g(x)\mu+m(x))\varphi,\quad \forall x\in\mathbb{R}.
$$
By the uniqueness of the principal eigenvalue, it follows that  $\lambda_m(-\mu)=\lambda_m(\mu), \forall \mu\in\mathbb{R}$. 

 In the case where $d\in C^{1+\nu}(\mathbb{R})(\nu\in(0,1))$ and $-g(x)=d'(x),\forall x\in\mathbb{R}$, for any given $\mu\in \mathbb{R}$, let $\psi(x)$ and $\phi(x)$ be the positive and $L$-periodic eigenfunctions
associated with  $\lambda_m(\mu)$ and $\lambda_m(-\mu)$, respectively, that is,
$$
(d(x)\psi')'-2\mu d(x)\psi'+(d(x)\mu^2-d'(x)\mu+m(x))\psi=\lambda_m(\mu)\psi
$$ 
and 
$$
(d(x)\phi')'+2\mu d(x)\phi'+(d(x)\mu^2+d'(x)\mu+m(x))\phi=\lambda_m(-\mu) \phi.
$$  
Using integration by parts, we have $$\int^{L}_{0}(d(x)\psi'(x))'\phi(x) dx=\int^{L}_{0}(d(x)\phi'(x))'\psi(x) dx,$$ and 
\begin{small}\begin{align*}-&\mu\int^L_0[2d(x)\psi'(x)\phi(x)+d'(x)\psi(x)\phi(x)]dx\\
&=\mu \int^L_0[2(d(x)\phi(x))'\psi(x)-d'(x)\psi(x)\phi(x)]dx\\
&=\mu\int^L_0[2d(x)\phi'(x)\psi(x)+d'(x)\phi(x)\psi(x)]dx.
\end{align*}
\end{small}
It then follows that \begin{eqnarray}\label{22}
\lambda_m (\mu)\int_0^L\psi(x)\phi(x)dx=\lambda_m (-\mu)\int_0^L\phi(x)\psi(x)dx.
\end{eqnarray} Since $\int_0^L\psi(x)\phi(x)dx>0$, we have $\lambda_m(\mu)=\lambda_m(-\mu), \forall \mu\in\mathbb{R}$. 
\end{proof}
\begin{lemma}\label{H45}
Assume that (H1) and (H2) hold. Then (H4) and (H5) are valid  provided that either all the coefficient functions of system \eqref{NModel} are even except $g_i$ is odd, or $d_i\in C^{1+\nu}(R)(\nu\in(0,1))$ and $g_i(x)=-d'_i(x), \forall x\in\mathbb{R}, i=1,2.$
\end{lemma}
\begin{proof}
First, we prove that (H4) holds. Indeed, in either case, by Lemma \ref{mp}(c) with $m(x)=b_1(x)$ and $d(x)=d_1(x)$, it is easy to see that the principle $\lambda_1(\mu)$ of \eqref{eep} is an even function of $\mu$ on $\mathbb{R}$. Since $\lambda_1(\mu)$ is convex on $\mathbb{R}$ and $\lambda_1(0)>0$, we have $\lambda_1(\mu)>0,\forall \mu>0.$ It follows that $c_{1+}^*=\inf_{\mu>0}\frac{\lambda _1(\mu)}{\mu}>0$. Similarly, we can show that $c_{2-}^*>0$, this implies $c_{1+}^*+c_{2-}^*>0$.

To verify (H5), it suffices to show that $\lim_{\mu\to 0^+}\frac{\lambda _2(\mu)}{\mu}=0$, where $\lambda_2(\mu)$ is the principal eigenvalue of \eqref{eep2}. In the case where all the coefficient functions of \eqref{NModel} are even except $g_i$ is odd, $i=1,2$, we have 
$$d_2(x)u_2^{*\prime\prime}(x)+g_2(x)u^{*\prime}_2(x)+u^*_2(x)(b_2(x)-a_{22}(x)u^*_2(x))=0, \quad x\in\mathbb{R}.$$
Let $u_2(x)=u^*_2(-x)$. Since $d_2$, $b_2$, $a_{22}$ are even and $g_2$ is odd, it follows that
$$d_2(x)u_2^{\prime\prime}(x)+g_2(x)u^{\prime}_2(x)+u_2(x)(b_2(x)-a_{22}(x)u_2(x))=0,\quad x\in\mathbb{R}.$$
This implies that $u^*_2(-x)$ is also an $L$-periodic positive steady state for scalar equation \eqref{VLseq} with $d(x)=d_2(x)$, $g(x)=g_2(x)$, $c(x)=b_2(x)$ and $e(x)=a_{22}(x), \forall x\in\mathbb{R}$. In view of Proposition \ref{VLexistence}, the uniqueness of the $L$-periodic positive steady state implies that $u^*_2(-x)=u^*_2(x),\forall x\in\mathbb{R}$. Taking $d(x)=d_2(x)$, $m(x)=b_2(x)-a_{22}(x)u^*_2(x)$, and $g(x)=g_2(x)$, or $g(x)=-d'_2(x)$ in \eqref{geep}, we see from Lemma \ref{mp}(c) that in two cases, $\lambda_2(\mu)$ is an even function on $\mathbb{R}$, and hence, $\lambda'_2(0)=0$. Since $\lambda_2(0)=0$, it follows that
$\lim_{\mu\rightarrow 0^+}\frac{\lambda _2(\mu)}{\mu}=\lambda_2'(0)=0<c^*_{1+}.$
\end{proof}
Now we impose the following assumption on system \eqref{dhmp}:
\begin{enumerate}
\item[(M)] $a(x)$ is non-constant, and $\overline{a}=\frac{1}{L}\int_{0}^{L}a(x)dx\ge0$. 
\end{enumerate}
\begin{lemma}\label{M}
Let (M) hold. Then (H1)--(H3) are valid for system \eqref{dhmp}. 
\end{lemma}
\begin{proof}
Let $\phi$ be the positive periodic eigenfunction associated with the principal eigenvalue $\lambda(d_1,0,a)$, that is,
$$d_1\phi''+a(x)\phi=\lambda(d_1,0,a)\phi.$$
Dividing the above equation by $\phi$ and integrating by parts on $[0,L]$, we get
$$\lambda(d_1,0,a)=\frac{1}{L}\int_{0}^{L}a(x)dx+d_1\int_{0}^{L}\left[\frac{\phi'(x)}{\phi(x)}\right]^2dx.$$  
Since $a(x)$ is non-constant, a simple computation shows that $\phi(x)$ is also non-constant. Therefore, we have $$\lambda(d_1,0,a)>\frac{1}{L}\int_{0}^{L}a(x)dx\ge0.$$
Similarly, we can show that $\lambda(d_2,0,a)>0$. 
It follows that (H1) holds, and hence, system \eqref{dhmp}  has three $L$-periodic steady states $E_0:=(0,0)$, $E_1:=(u^*_1(x),0)$ and $E_2:=(0,u^*_2(x))$ in $\mathbb{P}_+$. Note that 
\begin{equation}\label{eq}
d_2 u^{*''}_2(x)+u^*_2(x)(a(x)-u^*_2(x))=0,\quad x\in\mathbb{R}.
\end{equation}
It follows that $\lambda(d_2,0,a-u^*_2)=0$. If $a(x)-u^*_2(x)$ is a constant, then a straightforward computation shows that $u^*_2$ must be a positive constant eigenfunction associated with $\lambda(d_2,0,a-u^*_2)$. Therefore, $a(x)$ is also a constant, a contradiction. 

Note that for the eigenvalue problem \eqref{VLpep1} with $d_1(x)=d>0$ and $g\equiv0$, we have the variational formula for the principle eigenvalue (see, e.g, \cite{Hamel2}):
$$\lambda(d,0,h)=\min_{\phi\in E}\frac{-d\int^{L}_{0}[\phi'(x)]^2dx+\int^{L}_{0}h(x)\phi^2(x)dx}{\int^{L}_{0}\phi^2(x)dx},$$ where $E:=\{\phi\in C^2(\mathbb{R}):\phi(x)=\phi(x+L)>0,\ \forall x\in\mathbb{R}\}$. It easily follows that if $h(x)$ is non-constant, then $\lambda(d_1,0,h)>\lambda(d_2,0,h)$ provided $d_2>d_1>0$. Therefore, we have $\lambda(d_1,0,a-cu^*_2)>\lambda(d_2,0,a-u^*_2)=0$, that is, (H2) is valid for $c\in[0,1]$. To verify (H3), we suppose, by contradiction, that there is an $L$-periodic coexistence steady state $(u_0,v_0)\gg0$ in $\mathbb{P}_+$. Then we have 
\begin{eqnarray*}
& &d_1 u^{''}_0(x)+u_0(x)(a(x)-u_0(x)-cv_0(x))=0,\quad x\in\mathbb{R},\\
& &d_2 v^{''}_0(x)+v_0(x)(a(x)-u_0(x)-v_0(x))=0,\quad x\in\mathbb{R}.
\end{eqnarray*} 
This implies that  $\lambda(d_1,0,a-u_0-cv_0)=\lambda(d_2,0,a-u_0-v_0)=0$. By way of contradiction, we further show that $a-u_0-cv_0$ is non-constant, $\forall c\in[0,1]$. It then follows that $$\lambda(d_1,0,a-u_0-cv_0)>\lambda(d_2,0,a-u_0-cv_0)\ge\lambda(d_2,0,a-u_0-v_0), \forall c\in[0,1],$$ a contradiction. 
\end{proof}

As a consequence of Lemma \ref{M} and Theorem \ref{VLEQ}, we have the following result.
\begin{theorem}
Let (M) hold. Then $E_1:=(u^*_1(x),0)$ is globally asymptotically stable for all
initial values in $\mathbb{P}_+\backslash\{0,E_2\}$.
\end{theorem}   
For simplicity, we transfer system \eqref{dhmp} into the following cooperative system:
\begin{eqnarray}\label{dhmp2}
   & &\frac{\partial u_1}{\partial t}=d_1 \frac{\partial^2 u_1}{\partial x^2}+u_1(a(x)-cu^*_2(x)-u_1+cu_2), \\ 
   & &\frac{\partial u_2}{\partial t}=d_2\frac{\partial^2 u_2}{\partial x^2}+u_1(u^*_2(x)-u_2)
   +u_2(a(x)-2u^*_2(x)+u_2),\quad t>0, \ x\in\mathbb{R}. \nonumber
   \end{eqnarray}
   Let $u^*(\cdot)=(u^*_1(\cdot),u^*_2(\cdot))$. Define a family of operators  $\{Q_t\}_{t\ge0}$ on $\mathcal{C}_{u^*}$ by  $Q_t(\phi):=u(t,\cdot,\phi)$, where $u(t,\cdot,\phi)$ is the unique solution of system \eqref{dhmp2} with $u(0,\cdot)=\phi\in\mathcal{C}_{u^*}$. Let $\{\hat Q_t\}_{t\ge0}$ be defined as in \eqref{hatQ} and $\overline{c}_+$ be denoted by \eqref{defc} with $\tilde P=\hat Q_1$. By virtue of Lemma \ref{mp}, Lemma \ref{H45} and Proposition \ref{lb}, we see that $\overline{c}_+\ge c^0_+>0$. 
   
   The next result is the consequence of Theorem \ref{Qspreading} and Remark \ref{Re1}.
   \begin{theorem}
   Assume that (M) holds. Let $u(t,\cdot,\phi)$ be the solution of system \eqref{dhmp2} with $u(0,\cdot)=\phi\in\mathcal{C}_{u^*}$. Then the following statements are valid for system \eqref{dhmp2}:
   \begin{enumerate}
   			\item[(i)]If $\phi\in\mathcal{C}_{\beta}$, $0\le \phi\le \omega\ll \beta$ for some $\omega\in \mathcal{C}^{per}_{\beta}$, and $\phi(x)=0, \forall x\ge H$, for some $H\in \mathbb{R}$, then $\lim_{t\rightarrow\infty,x\ge ct}u(t,x,\phi)=0$ for any $c>\overline{c}_+$.
   			\item[(ii)]If $\phi\in\mathcal{C}_{\beta}$ and $\phi(x)\ge \sigma$, $\forall x\le K$, for some $\sigma\in \mathbb{R}^2$ with $\sigma\gg0$ and $K\in\mathbb{R}$, then $\lim_{t\rightarrow\infty,x\le ct}(u(t,x,\phi)-\beta(x))=0$ for any $c\in(0,\overline{c}_+)$.
   	\end{enumerate}
   \end{theorem}
In view of Theorem \ref{MIN}, we have the following result on periodic traveling waves for system \eqref{dhmp}. 
\begin{theorem}
Let (M) hold.
Then for any $c\ge\overline{c}_+$, system \eqref{dhmp} has an L-periodic rightward traveling wave $(U(x-ct,x),V(x-ct,x))$ connecting $(u^*_1(x),0)$ to $(0,u^*_2(x))$ with the wave profile component $U(\xi,x)$  being continuous and non-increasing in $\xi$, and $V(\xi,x)$ being continuous and non-decreasing in $\xi$. While for any $c\in(0,\overline{c}_+)$, system \eqref{dhmp} admits no $L$-periodic rightward traveling wave connecting $(u^*_1(x),0)$ to $(0,u^*_2(x))$.
\end{theorem}

It is not easy to verify conditions (D1) and (D2). However, motivated by \cite{SKT, SK,LLM}, we can formally compute  the lower bound $c^0_+$ in the case where  $$d_1=1, d_2>1, a(x)=\left\{\begin{array}{ll}
1, & ml< x<ml+l_1,\\
a<1, & ml-l_2\le x<ml,\quad  m\in\mathbb{Z},
\end{array}\right.$$ for system \eqref{dhmp} with $l=l_1+l_2$ and $\overline{a}=\frac{l_1+al_2}{l}>0$.
It is easy to see that 
$u^*_1(x)\approx a(x), u^*_2(x)\approx a(x)$, and hence, \eqref{eep0} becomes
\begin{align}\label{neep}
&  \lambda \psi\!=\!\psi''\!-\!2\mu\psi'\!+\!(\mu^2\!+\!(1-c))\psi,\quad ml<x<ml+l_1,\nonumber\\
& \lambda \psi\!=\!\psi''\!-\!2\mu\psi'\!+\!(\mu^2\!+\!a(1-c))\psi,\quad ml+l_1<x<(m+1)l.
\end{align}
The matching conditions are 
$$\lim_{x\rightarrow ml^-}\psi(x)=\lim_{x\rightarrow ml^+}\psi(x),\lim_{x\rightarrow (ml+l_1)^-}\psi(x)=\lim_{x\rightarrow (ml+l_1)^+}\psi(x),\quad m\in \mathbb{Z},$$
and
$$\lim_{x\rightarrow ml^-}\psi'(x)=\lim_{x\rightarrow ml^+}\psi'(x),\lim_{x\rightarrow (ml+l_1)^-}\psi'(x)=\lim_{x\rightarrow (ml+l_1)^+}\psi'(x),\quad m\in \mathbb{Z},$$
Set
\begin{eqnarray}
& &\phi(x)=A_1e^{\alpha_1x}+A_2e^{\alpha_2x}, \quad x\in[0,l_1],\\
& &\phi(x)=A_3e^{\beta_1(l-x)}+A_4e^{\beta_2(l-x)},\quad x\in[l_1,l],
\end{eqnarray}
where 
$\alpha_{1,2}=\mu\pm q_1$, $\beta_{1,2}=-\mu\pm q_2$, $q_1=\sqrt{\lambda-(1-c)}$, and $q_2=\sqrt{\lambda-a(1-c)}$.
Then the matching conditions yield the following linear relationship between the coefficients 
$$
\left(\begin{array}{cccc}
1 & 1&-1 &-1\\
e^{\alpha_1l_1}&e^{\alpha_2l_2}&-e^{\beta_1l_2}&-e^{\beta_2l_2}\\
q_1&-q_1 & q_2&-q_2\\
q_1e^{\alpha_1l_1}&-q_1e^{\alpha_2l_1} & q_2e^{\beta_1l_2}&-q_2e^{\beta_2l_2}
\end{array}\right)\left(\begin{array}{c}A_1\\A_2\\A_3\\A_4\end{array}\right)=0.
$$
Since we look for positive eigenfunctions, the determinant of the above matrix must be zero. Accordingly,
straightforward computations show that  
$$\cosh(\mu l)=\cosh(q_1l_1)\cosh(q_2l_2)+\frac{q^2_1+q^2_2}{2q_1q_2}\sinh(q_1l_1)\sinh(q_2l_2):= G(\lambda).$$
In view of $$\cosh^{-1}z=\log\{z+(z^2-1)^{1/2}\},\quad z>1,$$ we then have
$$\mu(\lambda)=\frac{1}{l}\log\{G(\lambda)+\sqrt{[G(\lambda)]^2-1}\}.$$
Let $\lambda_0$ be the solution of the following equation:
$$\frac{d\mu(\lambda)}{d\lambda}\frac{\lambda}{\mu(\lambda)}=1,$$
and $\mu_0=\mu(\lambda_0)$.
Thus, we obtain $c^0_+=\frac{\lambda_0}{\mu_0}$.

If $l\ll1$, by using $\cosh z\approx1+z^2/2$ and $\sinh z\approx z$, we get an approximation 
$$1+(\mu l)^2/2\approx(1+(q_1l_1)^2/2)(1+(q_2l_2)^2/2)+l_1l_2\frac{q^2_1+q^2_2}{2},$$
and hence,
$$c^0_+=\inf_{\mu>0}\frac{\lambda(\mu)}{\mu}\approx\inf_{\mu>0}\left\{\mu+\frac{(1-c)\overline{a}}{\mu}\right\}.$$
It follows that $c^0_+\approx2\sqrt{(1-c)\overline{a}}$,
$\mu_0\approx\sqrt{(1-c)\overline{a}}$, $\overline{a}=\frac{l_1+al_2}{l}>0$.

\section{Appendix}
In this section, we  extend the abstract results in \cite{FZ} and \cite{Liang2} on spreading speeds and traveling waves to the case of a periodic habitat. 

Let $\mathcal{C}$ be the set of all bounded and continuous functions
from $\mathbb{R}$ to $\mathbb{R}^m$ with $m\ge1$ and $\mathcal{C}_+=\{\phi\in\mathcal{C}:\phi(x)\geq0,\ \forall x\in \mathbb{R}\}$. Clearly, any vector in
$\mathbb{R}^m$ can be regarded as a function in $\mathcal{C}$. For
$u=(u_1,...,u_m), w=(w_m,...,w_m)\in \mathcal{C}$, we write $u\ge w
(u\gg w)$ provided $u_j(x)\ge w_j(x)(u_j(x)>w_j(x)), \forall 1\leq
j\leq m,\, x\in \mathbb{R}$, and $u>w$ provided $u\ge w$ but $u\neq
w$. Assume that $\beta$ is a  strongly positive $L$-periodic continuous function from $\mathbb{R}$ to $\mathbb{R}^m$. Set\begin{small}$$\mathcal{C}_{\beta}=\{u\in \mathcal{C}:\, 0\le u(x)\le \beta(x),\ \forall x\in \mathbb{R}\},\  \mathcal{C}^{per}_{\beta}=\{u\in \mathcal{C_\beta}:\,  u(x)=u(x+L),\ \forall x\in \mathbb{R}\}.$$ \end{small}
Let $X=C([0,L],\mathbb{R}^m)$ equipped with the maximum norm $|\cdot|_X$, $X_+=C([0,L],\mathbb{R}_+^m)$, $$X_{\beta}=\{u\in X:\ 0\le u(x)\le{\beta}(x),\  \forall x\in[0,L]\},\ \text{and} \  \overline{X}_{\beta}=\{u\in X_{\beta}: u(0)=u(L)\}.$$ Let $BC(\mathbb{R}, X)$ be the set of all continuous and bounded functions from $\mathbb{R}$ to $X$. Then we define 
\begin{small}$$\mathcal{X}=\{v\in BC(\mathbb{R},X):v(s)(L)=v(s+L)(0),\forall s\in \mathbb{R}\},  \mathcal{X}_+=\{v\in \mathcal{X}:v(s)\in X_+,\forall s\in \mathbb{R}\}$$\end{small} 
and  
$$\mathcal{X}_{\beta}=\{v\in BC(\mathbb{R},X_{\beta}):v(s)(L)=v(s+L)(0),\forall s\in \mathbb{R}\}.$$Let $$\mathcal{K}_{\beta}:=\{v\in BC(L\mathbb{Z},X_{\beta}):v(i)(L)=v(i+L)(0),\forall i\in L\mathbb{Z}\}.$$ Clearly, any element in $\overline{X}_{\beta}$ can be regarded as a constant function in $\mathcal{X}_{\beta}$, that is, any element in $\mathcal{C}^{per}_\beta$ corresponds to a constant function in $\mathcal{X}_{\beta}$. We
equip $\mathcal{C}$ and $\mathcal{X}$ with the compact open topology, that is, $u_n\to
u$ in $\mathcal{C}$ or $\mathcal{X}$ means that the sequence of $u_n(s)$ converges to $u(s)$ in $\mathbb{R}^m$ or $X$ uniformly for $s$ in any compact set. We equip $\mathcal{C}$ and $\mathcal{X}$ with the norm $\|\cdot\|_\mathcal{C}$ and $\|\cdot\|_\mathcal{X}$, respectively, by \[\|u\|_{\mathcal{C}}=\sum\limits_{k=1}^{\infty}\frac{\max_{|x|\le k}|u(x)|}{2^k},\ \forall u\in\mathcal{C},\] 
where $|\cdot|$ denotes the usual norm in $\mathbb{R}^m$, and
\[\|u\|_{\mathcal{X}}=\sum\limits_{k=1}^{\infty}\frac{\max_{|x|\le k}|u(x)|_X}{2^k},\ \forall u\in\mathcal{X}.\]

Define a translation operator $\mathcal{T}_a$ by
$\mathcal{T}_a[u](x)=u(x-a)$ for any given $a\in L\mathbb{Z}$. 
Let $Q$ be a operator on $\mathcal{C}_{\beta}$, where $\beta\in Int(\mathcal{C}_+)$ is $L$-periodic. In order to use the
theory developed in \cite{FZ} and \cite{Liang2}, we need the following
assumptions on $Q$:
\begin{enumerate}
	\item[(A1)] $Q$ is $L$-periodic, that is, $\mathcal{T}_a[Q[u]]
	=Q[\mathcal{T}_a[u]],\quad  \forall u\in \mathcal{C}_{\beta},\,
	a\in L\mathbb{Z}$.
	\item[(A2)] $Q:\, \mathcal{C}_{\beta} \to \mathcal{C}_{\beta}$ is continuous
	with respect to the compact open topology.
	\item[(A3)] $Q:\, \mathcal{C}_{\beta} \to \mathcal{C}_{\beta}$ is monotone
	(order preserving) in the sense that $Q[u] \ge
	Q[w]$ whenever $u \ge w$.
	\item[(A4)] $Q$
	admits two $L$-periodic fixed points $0$ and $\beta$ in $\mathcal{C}_+$, and for any $z\in \mathcal{C}^{per}_{\beta}$ with $0\ll z\leq \beta$, there holds $\lim\limits_{n\rightarrow\infty }Q^n[z](x)=\beta(x)$ uniformly for $x\in \mathbb{R}$.
	\item[(A5)] $Q[\mathcal{C}_{\beta}]$ is precompact in $\mathcal{C}_{\beta}$ with respect to the compact open topology.
\end{enumerate}
Define a homeomorphsim $F:\mathcal{C}_\beta\rightarrow \mathcal{K}_{\beta}$ by $$F[\phi](i)(\theta)=\phi(i+\theta),\ i\in L\mathbb{Z},\ \theta\in[0,L],$$ and a sequence of operators $P:\mathcal{K}_{\beta}\rightarrow \mathcal{K}_{\beta}$, by
 \begin{equation}\label{P}
 P=F\circ Q\circ F^{-1}.
 \end{equation}
Next, we define $\tilde{P}: \mathcal{X}\rightarrow \mathcal{X}$ by
\begin{equation}\label{P2}
\tilde{P}[v](s):=P[v(\cdot+s)](0),\quad \forall v\in\mathcal{X},\ s\in\mathbb{R}.
\end{equation}
We further claim that 
\begin{equation}\label{hatQ}
\tilde{P}[v](s)(\theta)=Q[v_s](\theta),\quad \forall v\in\mathcal{X},\ s\in\mathbb{R},\ \theta\in[0,L],
\end{equation}
where $v_s\in\mathcal{C}$ is defined by  
$$v_s(x)=v(s+n_x)(\theta_x),\quad \forall x=n_x+\theta_x\in\mathbb{R},\ n_x=L\left[\frac{x}{L}\right],\ \theta_x\in[0,L).$$
Indeed, since $$F[\phi](i)(\theta)=\phi(i+\theta),\quad F^{-1}[\psi](x)=\psi(n_x)(\theta_x),$$
it then follows that 
\begin{eqnarray*}
\tilde{P}[v](s)&=P[v(\cdot+s)](0)=FQF^{-1}[v(\cdot+s)](0)\\
&=F[Q[v(n_\cdot+s)(\theta_\cdot)]](0)=F[Q(v_s)](0),
\end{eqnarray*}
and hence,
\[\tilde{P}[v](s)(\theta)=F[Q(v_s)](0)(\theta)=Q[v_s](\theta).\]

Let $r\in Int(X_+)$ with $r(0)=r(L)$. In order to apply the results in \cite{FZ} to $\tilde{P}$, we need to verify that $
\tilde{P}$ satisfies the following assumptions:
\begin{enumerate}
	\item[(B1)]$\mathcal{T}_a[\tilde P[u]]
	=\tilde P[\mathcal{T}_a[u]],\quad  \forall u\in \mathcal{X}_{r},\,
	a\in \mathbb{R}$.
	\item[(B2)] $\tilde{P}:\, \mathcal{X}_{r} \to \mathcal{X}_{r}$ is continuous
	with respect to the compact open topology.

	\item[(B3)] $\tilde{P}:\, \mathcal{X}_{r} \to \mathcal{X}_{r}$ is monotone
	(order preserving) in the sense that $\tilde{P}[u] \ge
	\tilde{P}[w]$ whenever $u \ge w$.
	\item[(B4)] $\tilde{P}$
	admits two fixed points $0$ and $r$ in $\overline X_{r}$, and for any $z\in \overline{X}_{r}$ with $0\ll z\leq r$, 
	there holds $\lim\limits_{n\rightarrow\infty }{\tilde{P}}^n[z]=r$.
	\item[(B5)] There exists $k\in[0,1)$ such that for any $\mathcal{U}\subset\mathcal{X}_r$, $\alpha(\tilde P[\mathcal{U}](0))\le k\alpha(\mathcal{U}(0))$, where $\alpha$ denotes the Kuratowski measure of nonconmpactness in $\mathcal{X}_r$.
\end{enumerate} 
\begin{proposition}\label{HQASS}Let $\beta\in Int(\mathcal{C}_+)$ be $L$-periodic. Assume that $Q:\mathcal{C}_{\beta} \to \mathcal{C}_{\beta}$ satisfies assumptions (A1)--(A5). Then $\tilde{P}$: $\mathcal{X}_{ \beta}\rightarrow\mathcal{X}_{\beta}$  satisfies assumptions (B1)--(B5).
\end{proposition}
\begin{proof}

For any $c\in \mathbb{R}$, let $u(\cdot)=v(\cdot+c), \forall v\in\mathcal{X}$. Then
	\begin{eqnarray*}T_{-c}\tilde{P}[v](s)&=&\tilde{P}[v](s+c)\\
		&=&Q[v_{s+c}]=Q[u_{s}]=\tilde{P}[u(\cdot)](s)\\
		&=&\tilde{P}[T_{-c}v](s), \quad \forall v\in\mathcal{X},\ s\in\mathbb{R},
	\end{eqnarray*}
	and hence, (B1) holds. (B2) can be verified by similar arguments to those in \cite[Lemma 2.1]{Liang}, and (B3) directly follows from (A3). Clearly, $0$ is the fixed point of $\tilde{P}$ since $Q(0)=0$. To verify (B4), we need to show that $\beta|_{[0,L]}$ is the fixed point of $\tilde{P}$. Note that $\beta(x)$ is a constant function in $\mathcal{X}$ with $x\in[0,L]$ we have 
		$$\beta_s(\cdot)=\beta(s+n_\cdot)(\theta_\cdot)=\beta(\theta_\cdot),\quad  \forall s \in \mathbb{R}.$$ Therefore, $\beta_s=\beta$ in $\mathcal{C}, \forall s\in\mathbb{R}$. Moreover,
		\begin{eqnarray*}
			\tilde{P}[\beta](s)(\theta)=Q[ \beta_s](\theta)=Q[\beta](\theta)=\beta(\theta),\quad \forall \theta\in[0,L].
		\end{eqnarray*}   
		This implies that $\tilde{P}[\beta]=\beta$ in $\mathcal{X}$. Thus, $(B4)$ follows from $(A4)$. 
		Now we prove (B5) holds. For any given $\mathcal{U}\subset\mathcal{X}_{\beta}$, it is easy to see that $\tilde P(\mathcal{U})(0)$ is uniformly bounded. By (A5), it follows for any  $\varepsilon>0$, there exists $\delta>0$ such that $$|Q(v)(x_1)-Q(v)(x_2)|<\varepsilon, \quad\forall v\in \mathcal{C}_\beta$$ provided that $x_1,x_2\in[0,L]$ with $|x_1-x_2|<\delta$.  So for any $v\in\mathcal{U}$,
		\[|\tilde P(v)(0)(\theta_1)-\tilde P(v)(0)(\theta_2)|=|Q(v_0)(\theta_1)-Q(v_0)(\theta_2)|<\varepsilon\]
		provided that $\theta_1,\theta_2\in[0,L]$ with $|\theta_1-\theta_2|<\delta$.
		This implies that $\tilde P(\mathcal{U})(0)$ is equicontinuous. By Arzel\`{a}--Ascoli theorem, it follows that $\tilde P(\mathcal{U})(0)$ is precompact in $X_{\hat{\beta}}$, and hence,  $\alpha(\tilde P(\mathcal{U})(0))=0$, this proves (B5) with $k=0$.
\end{proof}

Let $\omega\in \overline{X}_\beta $ with $0\ll\omega\ll \beta$. Choose $\phi \in\mathcal{X}_\beta$ such that the following properties hold:
\begin{enumerate}
	\item[(C1)]$\phi(s)$ is nonincreasing in $s$;
	\item[(C2)] $\phi(s)\equiv0$ for all $s\ge0$;
	\item[(C3)]$\phi(-\infty)=\omega$. 
\end{enumerate}
Let $c$  be a given real number. According to \cite{Wein82}, we define an operator $R_{c}$ by
\[R_c[a](s):=\max\{\phi(s),T_{-c}\tilde{P}[a](s)\},\]
and a sequence of functions $a_n(c;s)$ by the recursion:
$$a_0(c;s)=\phi(s),\quad a_{n+1}(c;s)=R_{c}[a_n(c;\cdot)](s).$$
As a consequence of  similar arguments to those in \cite[Lemmas 3.1--3.3]{FZ}, we have the following result.
\begin{lemma}
	The following statements are valid:
	\begin{enumerate}
		\item[(1)]For each $s\in \mathbb{R}$, $a_n(c,s)$ converges to $a(c;s)$ in $X$, where $a(c;s)$ is nonincreasing in both $c$ and $s$, and $a(c;\cdot)\in\mathcal{X}_{\beta}$. 
		\item[(2)]$a(c,-\infty)= \beta$ and $a(c,+\infty)$ existing in $X$ is a fixed point of $\tilde{P}$. 
	\end{enumerate}
\end{lemma} 
Following \cite{WLL2002,FZ}, we define two numbers 
\begin{eqnarray}\label{defc}
c^*_+=\sup\{c:a(c,+\infty)=\beta\},\quad \overline c_+=\sup\{c:a(c,+\infty)>0\}
\end{eqnarray}
Clearly, $c^*_+\le \overline{c}_+$ due to the monotonicity of $a(c;\cdot)$ with respect to $c$. For each $t\ge 0$. Let $P_t$ and $\tilde{P}_t$ be defined as in \eqref{P} and \eqref{hatQ} with $Q=Q_t$, respectively. 
By \cite[Remark 3.2]{FZ}, we have the following result. 
\begin{theorem}\label{spreading}Let $\{Q_t\}_{t\ge0}$ be a continuous-time semifow on $\mathcal{C}_\beta$ with $Q_t[0]=0,Q_t[\beta]=\beta$ for all $t\ge0$ and  $\{\tilde{P}_t\}_{t\ge0}$ be defined as in \eqref{hatQ} for each $t\ge0$, and $c^*_+$ and $\overline{c}_+$ be denoted by \eqref{defc} with $\tilde{P}=\tilde{P}_1$. Suppose that $Q_t$ satisfies (A1)--(A5) for each $t>0$. Then the following statements are valid:
	
	\begin{enumerate}
		\item[(i)]If $\phi\in\mathcal{C}_{\beta}$, $0\le \phi\le \omega\ll \beta$ for some $\omega\in \mathcal{C}^{per}_{\beta}$, and $\phi(x)=0, \forall x\ge H$, for some $H\in \mathbb{R}$, then $\lim_{t\rightarrow\infty,x\ge ct}Q_t(\phi)=0$ for any $c>\overline{c}_+$.
		\item[(ii)]If $\phi\in\mathcal{C}_{\beta}$ and $\phi(x)\ge \sigma$, $\forall x\le K$, for some $\sigma\gg0$ and $K\in\mathbb{R}$, then $\lim_{t\rightarrow\infty,x\le ct}(Q_t(\phi)(x)-\beta(x))=0$ for any $c<c^*_+$.
	\end{enumerate}
\end{theorem}
\begin{proof}
Since $\{Q_t\}_{t\ge0}$ is a continuous-time semifow on $\mathcal{C}_\beta$ with $Q_t(0)=0$ and $Q_t(\beta)=\beta$ for all $t\ge0$, it follows that  $\{\tilde{P}_t\}_{t\ge0}$ is a continuous-time semiflow on $\mathcal{X}_{\beta}$ with $\tilde{P}_t(0)=0$ and $\tilde{P}_t(\beta)=\beta$ for all $t\ge0$. By Proposition \ref{HQASS}, $\tilde{P}_t$ satisfies (B1)--(B5).
For any $\phi\in\mathcal{C}_{\beta}$, $0\le \phi\le \omega\ll \beta$ with $\omega\in \mathcal{C}^{per}_{\beta}$, let 
\[u(s)(\theta)=[\phi(n_s+L+\theta)-\phi(n_s+\theta)]\theta_s+\phi(n_s+\theta).\]
for $s\in\mathbb{R}$, $s=n_s+\theta_s$, $n_s=L\left[\displaystyle\frac{s}{L}\right],\ \theta_s\in[0,L)$, $\theta\in[0,L]$. Then $u\in \mathcal{X}_{{\beta}}$, $0\le u\le \omega\ll \beta$. 

To prove statement $(i)$, we suppose that there exists some $H\in\mathbb{R}$ such that $\phi(x)=0$, $x\ge H$ and $\phi(x)\not\equiv0$ (otherwise, it is trivial). Thus, $u(s)=0$, $s\ge H+L$. By \cite[Remark 3.2]{FZ},
it follows that $\lim_{t\rightarrow\infty,s\ge ct}\tilde{P}_t(u)(s)=0$ in $X$ for any $c>\overline{c}_+$. 
On the other hand, we have 
\begin{eqnarray*}\tilde{P}_t[u](n_x)(\theta_x)&= & Q_t[u_{n_x}](\theta_x)=Q_t[u(n_x+n_\cdot)(\theta_\cdot)](\theta_x),\\
&=&Q_t[\phi(n_x+\cdot)](\theta_x)=Q_t[\phi(\cdot)](x),\quad x\in\mathbb{R},
\end{eqnarray*}
and for $s\in L\mathbb{Z}$,   $\lim_{t\rightarrow\infty,s\ge ct}\tilde{P}_t(u)(s)=0$ in $X$ holds true for any $c>\overline{c}_+$. Choose a $c'\in(\overline c_+,c)$, we obtain
\begin{equation} \label{ieq1}
|Q_t[\phi](x)|\le|\tilde{P}_t[u](n_x)|_ X,\quad \forall x\ge ct, \ t\ge \frac {L}{c-c'},
\end{equation} and $n_x\ge ct-L\ge c't$. Letting $t\rightarrow \infty$ in \eqref{ieq1}, we have $\lim_{t\rightarrow\infty,x\ge ct}Q_t(\phi)=0$ for any $c>\overline{c}_+$.

By similar arguments to the above, we can show that statement (ii) is also valid.
\end{proof}
In view of the above theorem, we may regard $\overline c_+$ and $c^*_+$, respectively, as the fastest and slowest rightward spreading speeds for $\{Q_t\}_{t\ge0}$ on $\mathcal{C}_\beta$. If $\overline c_+=c^*_+$, then we say that this system admits a single rightward spreading speed.

Next, we address the existence and non-existence of traveling waves in a periodic habitat for the continuous-time semiflow $\{Q_t\}_{t\ge0}$. Given a continuous-time semiflow $\{Q_t\}_{t\ge0}$ on $\mathcal{C}_\beta$, we say that $V(x-ct,x)$ is an $L$-periodic rightward traveling wave of $\{Q_t\}_{t\ge0}$ if $V(\cdot+a,\cdot)\in \mathcal{C}_\beta$,\ $\forall a\in \mathbb{R}$, $Q_t[U](x)=V(x-ct,x)$, $\forall t\ge0$, and $V(\xi,x)$ is an $L$-periodic function in $x$ for any fixed $\xi\in\mathbb{R}$, where $U(x):=V(x,x)$. Moreover, we say that $V(\xi,x)$ connects $\beta$ to $0$ if $\lim_{\xi\rightarrow -\infty}|V(\xi,x)-\beta(x)|=0$ and $\lim_{\xi\rightarrow +\infty}|V(\xi,x)|=0$ uniformly for $x\in\mathbb{R}$.

Since we have only shown the weak compactness (B5) for $\tilde{P}_t$, we cannot directly apply \cite[Theorem 4.1]{FZ} to $\{\tilde{P}_t\}_{t\ge0}$ on $\mathcal{X}_{\beta}$. However, $\{P_t\}_{t\ge0}$ on $\mathcal{K}_{\beta}$ has the compactness because any element in $\mathcal{K}_{\beta}$ is defined on the discrete domain. 
  Following the proof of Case 1 in \cite[Theorem 4.2]{Liang2} and the argument in \cite[Theorem 3.1]{FZ}, we obtain the existence and non-existence of traveling waves for the discrete-time dynamical system $\{P_1^n\}$ on $\mathcal{K}_{\beta}$. Thus, the existence and non-existence of traveling waves for the continuous-time dynamical system $\{P_t\}_{t\ge0}$ on $\mathcal{K}_{\beta}$ follows from the arguments in \cite[Theorem 4.4]{Liang2}. By similar arguments to those in \cite[Theorem 5.3]{Liang2}, we can extend \cite[Theorem 4.1]{FZ} to the case of a periodic habitat so that the following result holds true. 

\begin{theorem}\label{Tw}
Let $\{Q_t\}_{t\ge0}$ be a continuous-time semifow on $\mathcal{C}_\beta$ with $Q_t[0]=0,Q_t[\beta]=\beta$ for all $t\ge0$, $\{\tilde{P}_t\}_{t\ge0}$ be defined as in \eqref{hatQ}, and $c^*_+$ and $\overline{c}_+$  be denoted by \eqref{defc} with $\tilde{P}=\tilde{P}_1$. Suppose that $Q_t$ satisfies (A1)--(A5) for each $t>0$. Then the following statements are valid:
	\begin{enumerate}
	\item[(1)] For any $c\ge c^*_+$, there is an $L$-periodic  traveling wave $W(x-ct,x)$ connecting $\beta$ to some equilibrium $\beta_1\in C^{per}_\beta\backslash\{\beta\}$ with $W(\xi,x )$ be continuous and nonincreasing in $\xi\in \mathbb{R}$.
	\item[(2)]If, in addition, $0$ is an isolated equilibrium of $\{Q_t\}_{t\ge0}$ in $\mathcal{C}^{per}_\beta$, then for any $c\ge\overline{c}_+$ either of the following holds true:\begin{enumerate}
	 \item[(i)] there exists an $L$-periodic traveling wave $W(x-ct,x)$ connecting $\beta$ to $0$ with $W(\xi,x )$ be continuous and nonincreasing in $\xi\in \mathbb{R}$.
		 \item[(ii)]$\{Q_t\}_{t\ge0}$ has two ordered equilibria $\alpha_1$,$\alpha_2 \in C^{per}_\beta\backslash\{0,\beta\}$ such that there exist an $L$-periodic traveling wave $W_1(x-ct,x)$ connecting $\alpha_1$ and $0$ and an $L$-periodic traveling wave $W_2(x-ct,x)$ connecting $\beta$ and $\alpha_2$ with $W_i(\xi,x ), i=1,2$ be continuous and nonincreasing in $\xi\in \mathbb{R}$.
		\end{enumerate}
	\item[(3)] For any $c< c^*_+$, there is no $L$-periodic traveling wave connecting $\beta$, and for any $c<\overline{c}_+$, there is no $L$-periodic traveling wave connecting $\beta$ to $0$.  
	\end{enumerate}
\end{theorem}

\end{document}